\documentclass[10pt]{amsart}
\usepackage[margin=1in]{geometry}

\usepackage{amssymb,amsmath,amsthm,amscd,epsf,latexsym,verbatim,graphicx,amsfonts,hyperref,epstopdf,xcolor,wasysym}
\usepackage{mathtools}

\usepackage{thmtools} 
\usepackage{thm-restate}

\usepackage{wrapfig}
\usepackage{cutwin}
\usepackage{tikz}
\usepackage{tkz-graph}
\usepackage{caption}
\usepackage{float}
\usetikzlibrary{shapes,snakes}
\usetikzlibrary{arrows}
\usetikzlibrary{shapes.misc}

\usepackage{array}  
\usepackage{longtable} 
\usepackage{booktabs}
\newcolumntype{C}{>{$}c<{$}} 
\input{insbox}

\tikzset{cross/.style={cross out, draw=black, fill=none, minimum size=2*(#1-\pgflinewidth), inner sep=0pt, outer sep=0pt}, cross/.default={2pt}}

\theoremstyle{theorem}
\newtheorem{theorem}{Theorem}[section]

\newtheorem{lemma}[theorem]{Lemma}
\newtheorem{proposition}[theorem]{Proposition}

\theoremstyle{definition}

\theoremstyle{definition}

\theoremstyle{definition}
\newtheorem*{remark}{Remark}
\theoremstyle{definition}
\newtheorem*{example}{Example}
\theoremstyle{definition}
\newtheorem{definition}[theorem]{Definition}
\theoremstyle{definition}

\theoremstyle{definition}

\theoremstyle{definition}

\newcommand{\ZZ}{\mathbb{Z}}
\newcommand{\RR}{\mathbb{R}}

\newcommand{\CC}{\mathbb{C}}
\newcommand{\NN}{\mathbb{N}}
\newcommand{\TT}{\mathbb{T}}

\newcommand{\Prim}{\hbox{\sf RP}}

\newcommand{\STS}{\hbox{\sf STS}}
\newcommand{\RSTS}{\hbox{\sf rSTS}}
\newcommand{\PSTS}{\hbox{\sf pSTS}}
\newcommand{\calH}{\mathcal{H}}
\newcommand{\Hol}{\hbox{\sf Hol}}
\newcommand{\Veech}{SL}
\newcommand{\Unit}{\hbox{\sf Unit}}

\newcommand{\SL}{\mathrm{SL}}

\newcommand{\sdot}{\! \cdot \!}

\begin{document}
\title{Statistics of square-tiled surfaces: symmetry and short loops}

\author{Sunrose Shrestha}
\address{Department of Mathematics, Tufts University, 503 Boston Avenue, Medford, MA 02155}
\email{sunrose.shrestha@tufts.edu}

\author{Jane Wang}
\address{Department of Mathematics,  Massachusetts Institute of Technology, 77 Massachusetts Ave, Cambridge, MA 02139}
\email{janeyw@mit.edu}

%%
%% If there is another author uncomment and edit the following.
%%

%\author{Second Author}
%\address{Department of Mathematics, University of South Carolina,
%Columbia, SC 29208}
%\email{second@math.sc.edu}
%\urladdr{www.math.sc.edu/$\sim$second}

\begin{abstract}
Square-tiled surfaces are a class of translation surfaces that are of particular interest in geometry and dynamics because, as covers of the square torus, they share some of 
its simplicity and structure.  
In this paper, we study  counting problems that result from focusing on properties of the square torus one by one.
After drawing insights from experimental evidence, we consider the implications between these properties and as well as the frequencies of these properties in each stratum of translation surfaces.
\end{abstract} 
\maketitle

\section{Introduction}

Translation surfaces form an important class of metrics on two-manifolds, namely those that admit an atlas whose 
transition functions are given by Euclidean translation.  They can be viewed from several other, equivalent perspectives:
\begin{itemize}
\item complex analysis:  a translation surface is a pair $(X,\omega)$ where $X$ is a Riemann surface and
$\omega$ is an Abelian differential
(away from finitely many singular points, $\omega$  is the pullback of the one-form $dz$);
\item Euclidean geometry:  a translation surface is a collection of Euclidean polygons with sides glued by 
translation in parallel pairs.
\end{itemize}

\begin{wrapfigure}{R}{0.30\textwidth}
\vspace{-15pt}
%\begin{figure}[ht]
\centering
\begin{tikzpicture}
    \draw (0,0) -- (4,0); 
    \draw (0,1) -- (4,1);
    \draw (0,2) -- (2,2);
    \draw (0,0) -- (0,2);
    %\draw (0,1) -- (0,2);
    %\draw (1,1) -- (1,2);
    \draw (1,0) -- (1,2);
    \draw (2,0) -- (2,2);
    \draw (3,0) --(3,1);
    \draw (4,0) -- (4,1);
    
    \node at (2.5, 0) [below=.075cm] {$a$};
     \node at (3.5, 1) [above] {$a$};
     \node at (2.5, 1) [above] {$b$};
     \node at (3.5, 0) [below=0cm] {$b$};
     
     \foreach \i in {1,...,4}
{ \node at (\i-0.5, 0.5)  {\i};
}

    \foreach \i in {5,6}
{ \node at (\i-4.5, 1.5)  {\i};
}

     \draw [fill] (0,0) circle [radius=0.05];
     \draw [fill] (0,2) circle [radius=0.05];    
      \draw [fill] (2,0) circle [radius=0.05];
     \draw [fill] (2,2) circle [radius=0.05];   
      \draw [fill] (3,1) circle [radius=0.05];   
       \draw [fill] (4,0) circle [radius=0.05];

        \draw (0,1) node[cross=3] {};
        \draw (2,1) node[cross=3] {};
        \draw (3,0) node[cross=3] {};
        \draw (4,1) node[cross=3] {};

\end{tikzpicture}
\caption{A genus two surface from $\STS_6$. The labels $a,b$ indicate edge identifications (the others are identified with their opposites);  the vertices are marked
to indicate the induced identifications, which produce two cone points of angle $4\pi$, apart from which the surface is flat.
The permutation pair $\sigma=(1234)(56)$, $\tau=(15)(26)(34)$ describes the horizontal and vertical gluings.}
\label{fig:STS}
%\end{figure} 
\end{wrapfigure}
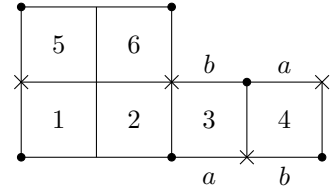

A square-tiled surface (or STS) is a special case of translation surface in which the Riemann surface $X$ covers
 the square torus $\TT=\CC/(\ZZ[i])$, branched over only one point; equivalently, the Euclidean polygons can be taken to be squares.
 Let the set of such surfaces made from $n$ squares be denoted by $\STS_n$. STSs have a third, combinatorial, description as a pair 
of permutations from $S_n$ describing the vertical and horizontal gluing of squares. 
 This is illustrated in Figure~\ref{fig:STS}.

One fundamental tool  is the action of $SL_2(\RR)$ by linear transformations on the space of translation surfaces.
Translation surfaces have been widely studied in connection to Teichm{\"u}ller theory.  
Teichm\"uller space parametrizes all Riemann surface structures on a fixed topological surface; each 
translation surface defines a one-parameter family in Teichm\"uller space via the action of the diagonal subgroup
of $SL_2(\RR)$, and these lines turn out to be geodesic in the famous Teichm\"uller metric.
%\footnote{In fact,
%if we slightly generalize abelian differentials to quadratic differentials by allowing translation-plus-reflection gluings,
%these diagonal orbits describe all Teichm\"uller geodesics.}  
Many distinguished authors, including Masur, Veech, Eskin, Mirzakhani, and Avila, have developed a formidable
literature on Teichm\"uller dynamics with this as a starting point.

Translation surfaces also naturally arise in the study of polygonal billiards. Given a polygonal billiard table with rational multiples of $\pi$ as the angles, one applies an unfolding construction due to Katok and Zemlyakov to associate a closed translation surface to the polygonal billiard. Billiard trajectories  become geodesic curves on the associated translation surface. 
The long-term behavior of such trajectories can then be explored by studying curves on translation surfaces. 

An aid to studying the symmetry of a translation surface is 
its stabilizer in $SL_2(\RR)$, called its {\em Veech group} or group of affine homeomorphisms.
A celebrated special case of translation surfaces is those with lattice Veech groups, sometimes known as {\em Veech surfaces}, 
which  are known to have very strong geometrical and dynamical properties (\cite{forestertangtao,smillieweiss,veech, hubertschmidt}). For example, Veech surfaces have so-called ``optimal dynamics": in any direction on the surface, each infinite trajectory is either periodic or dense. Square-tiled surfaces are always Veech surfaces since their Veech groups are  
commensurable to $SL_2(\ZZ)$.  

\subsection*{Counting problems and designer examples}
For any translation surface, we can define its set of {\em holonomy vectors} as the collection of
Euclidean displacement vectors between cone points, which belong to $\RR^2$.  
Square-tiled surfaces can be characterized as translation surfaces with holonomy vectors in $\ZZ^2$.
Because of this, STSs are sometimes thought of as the integer points in the space of translation surfaces. 
Translation surfaces are naturally stratified by the number and cone angle of singularities; for instance, 
the surface in Figure~\ref{fig:STS} lies in the stratum of genus-two surfaces with two simple cone points,
called $\calH(1,1)$. In general, the stratum $\calH(\alpha)$, where $\alpha = (\alpha_1, \dots, \alpha_s)$, denotes the collection of translation surfaces with $s$ singular points of angle $2\pi(\alpha_i+1)$ for $i=1, \dots, s$. These strata can be endowed with an $SL_2(\RR)$ invariant measure, and it is known that with respect to this measure, the strata of unit area translation surfaces have finite volume. 
Recall that in Euclidean space, one can estimate the volume of a ball of radius $r$
by the number of integer points inside it---the classic Gauss circle problem then asks for the correct bound $|E(r)|$ on the difference between this count and the volume of the ball. In exactly the same way, studying counting problems in $\STS_n$ for larger and larger $n$
gives volume estimates for strata of translation surfaces  (\cite{eskinokounkov,zorich}).
The many applications of such asymptotic theorems  
have also spurred some other counting problems, such as counting by orbits under $SL_2(\ZZ)$ (\cite{lelievreroyer}) or by 
the cylinder decomposition structure (\cite{delgouzogzor,shrestha}). 
Zmiaikou even uses STSs to compute the probability that two permutations, whose commutator is a product of two disjoint transpositions, generate the full symmetric group,
taking advantage of the permutation pair description above (\cite{zmiaikou}).

\begin{wrapfigure}{R}{0.30\textwidth}
\vspace{-5pt}
%\begin{figure}[ht]
\begin{tikzpicture}
    \draw (0,0) -- (4,0); 
    \draw (0,1) -- (4,1);
    \draw (0,2) -- (1,2);
    \draw (0,0) -- (0,2);
    %\draw (0,1) -- (0,2);
    \draw (1,0) -- (1,-1);
      \draw (2,0) -- (2,-1);
        \draw (1,-1) -- (2,-1);
    \draw (1,0) -- (1,2);
    \draw (2,0) -- (2,2);
    \draw (3,0) -- (3,2);
        \draw (3,0) -- (3,-1);
         \draw (3,-1) -- (4,-1);

     \draw (2,2) -- (3,2);
    \draw (4,0) -- (4,1);
 \draw (4,0) -- (4,-1);
     
     \draw[thin] (0.5,1.85) -- ++(0,0.3);
      \draw[thin] (2.5,-0.15) -- ++(0,0.3);
      
      \draw[thin] (0.45,-0.15) -- ++(0,0.3);
       \draw[thin] (0.55,-.15) -- ++(0,0.3);
       
       \draw[thin] (2+0.45,-0.15+2) -- ++(0,0.3);
       \draw[thin] (2+0.55,-.15+2) -- ++(0,0.3);
       
        \draw[thin] (3+0.40,-0.15-1) -- ++(0,0.3);
          \draw[thin] (3+0.50,-0.15-1) -- ++(0,0.3);
           \draw[thin] (3+0.60,-.15-1) -- ++(0,0.3);
           
            \draw[thin] (1+0.40,-0.15+1) -- ++(0,0.3);
          \draw[thin] (1+0.50,-0.15+1) -- ++(0,0.3);
           \draw[thin] (1+0.60,-.15+1) -- ++(0,0.3);

          \draw[thin] (1+0.50,-0.15-1) -- ++(0,0.3);
          \draw[thin] (1+0.35,-0.15-1) -- ++(0.3,0.3);
           %\draw[thin] (1+0.60,-.15+1) -- ++(0,0.3);

	\draw[thin] (3+0.50,-0.15+1) -- ++(0,0.3);
          \draw[thin] (3+0.35,-0.15+1) -- ++(0.3,0.3);

 	\draw[thin] (-0.15, 1.5) -- ++(0.3,0);
   	 \draw[thin] (4-0.15, 1.5-2) -- ++(0.3,0);
	 
	  \draw[thin] (1-0.15, 1+0.45) -- ++(0.3,0);
       \draw[thin] (1-.15, 1+0.55) -- ++(0.3,0);

 	\draw[thin] (1-0.15, -1+0.45) -- ++(0.3,0);
       \draw[thin] (1-.15, -1+0.55) -- ++(0.3,0);

  \draw[thin] (-0.15+2, 1+0.4) -- ++(0.3,0.);
          \draw[thin] (-0.15+2, 1+0.5) -- ++(0.3,0);
           \draw[thin] (-.15+2, 1+0.6) -- ++(0.3,0);
           
            \draw[thin] (-0.15+2, -1+0.4) -- ++(0.3,0.);
          \draw[thin] (-0.15+2, -1+0.5) -- ++(0.3,0);
           \draw[thin] (-.15+2, -1+0.6) -- ++(0.3,0);
           
           \draw[thin] (-0.15+3, 1+0.5) -- ++(0.3,0);
          \draw[thin] (0.15+3, 1+0.35) -- ++(-0.3,0.3);

 	 \draw[thin] (-0.15+3, -1+0.5) -- ++(0.3,0);
          \draw[thin] (0.15+3, -1+0.35) -- ++(-0.3,0.3);
\end{tikzpicture}
\captionof{figure}{The EW surface} 
\label{fig:pig}
\end{wrapfigure}
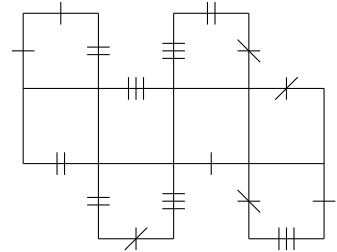

Furthermore, particular examples of STSs have sometimes been very useful for building theory.
The most famous example is a surface in $\STS_8$ that is colorfully called the {\em Eierlegendre Wollmichsau}, 
or ``the egg-laying wool-milk sow," because of its many surprising applications.  First explored by 
Herrlich and Schmithusen in \cite{herrlichschmit}, it can be seen below in Figure~\ref{fig:pig}, and we will refer to 
it as the EW surface.  %In that spirit, we will also present new examples of simple STSs which possess (or lack)
%some of the torus-like features we define below.

The EW surface has properties similar to the square torus $\TT$ in multiple ways, 
and in this paper we identify  some of these structural similarities and analyze them separately. 
We prove implications and non-implications between these properties and give experimental evidence suggesting
counting statements, with some rigorous finiteness results as well. We will also present new examples of simple STSs which possess (or lack) some of the torus-like features we explore. For instance, the set of holonomy vectors of 
the $\TT$ is the relatively prime pairs in $\ZZ^2$, denoted $\Prim$, and sometimes called visible points. So, for us, a \emph{visibility torus} is an STS for which $\Prim$ is contained in the set of holonomy vectors of $O$, denoted $\Hol(O)$. We obtain the following result, regarding the visibility properties of STSs in a fixed stratum.

\begin{restatable*}[Visibility] {theorem}{faketori}
\label{faketori}
For a fixed stratum, $\calH(\alpha) = \calH(\alpha_1, \dots, \alpha_s)$, the visibility properties of STSs in the stratum, is governed by their number of squares, $n$, in the following way:

\begin{figure}[h!!!!]
\begin{tikzpicture}
\node at (-3,0){$\calH(\alpha)$};
\draw[-latex]   (-2,0) -- (13,0) ; %edit here for the axis

\draw[shift={(-2,0)},color=black] (0,0.2) -- (0,-0.2) node[below] 
{$0$};

%\fill [shift={(-2,0)}, color=gray!10] (0.02,-0.1) rectangle (3.96,0.1);

%\draw[decorate,decoration={brace,amplitude=10pt}]
%(-1.9,0.2) -- (1.9,0.2);

%%%%%% open close intervals
\draw[{[-)}, thick] (-2,.4) -- (2, .4);
\draw[{(-]}, thick] (2,.4) -- (6, .4);
\draw[{(- >}, thick] (11,.4) -- (13, .4);

%%%%%%%% interval labels
\node[scale=0.7] at (-0.2,1){$\STS_n \cap \calH(\alpha) = \emptyset$};
\node[shift={(2,0)},scale=0.7]at (0, 1){$\Prim  = \Hol(O)$};
\node[scale=0.7] at (4,1){$\Prim \subsetneq  \Hol(O)$};
\node[scale=0.7] at (12,1){$\Prim - \Hol(O) \neq \emptyset$};

%%%%%%%%% misc
\draw[shift={(2,0)},color=black, -latex] (0,0.9)--(0,0.5);
\node[scale=0.7] at (8.35,0.8){};
\draw[shift={(-1,0)},-latex] (3,-1.5) -- (9.5,-1.5) node[midway,fill=white] {Number of squares};

%%%%%%%%%%% number labels
\draw[shift={(2,0)},color=black] (0,0.2) -- (0,-0.2) node[below] 
{$2g+s-2$} ;
\draw[shift={(6,0)},color=black] (0, 0.2) -- (0,-0.2) node[below] 
{$4g+2s-5$} ;
\draw[shift={(11,0)},color=black] (0,0.2) -- (0,-0.2) node[below] 
{$N(\alpha)$} ;

%%%%%% scrap

%\draw[decorate,decoration={brace,amplitude=10pt}, shift={(4,0)}]
%(-1.9,0.2) -- (1.9,0.2);

%\fill [shift={(2,0)}, color=gray!30] (0.04,-0.1) rectangle (3.98,0.1);

%\draw[decorate,decoration={brace,amplitude=10pt}, shift={(8,0)}]
%(-1.9,0.2) -- (2.9,0.2);

%\fill [shift={(6,0)}, color=gray!70] (0.02,-0.1) rectangle (4.98,0.1);

%\fill [shift={(11,0)}, color=gray] (0.02,-0.1) rectangle (1.8,0.1);

%\draw[decorate,decoration={brace,amplitude=10pt}, %shift={(12,0)}]
%(-0.9,0.2) -- (0.9,0.2);
%\draw[*-o] (0.92,0) -- (2.08,0);
%\draw[very thick] (0.92,0) -- (1.92,0);

%\node at (6.5,-1.5){number of squares};

%following for legend.

%\fill [gray!10] (0,-2.5) rectangle (0+0.5,-2.5+0.5);
%\node at (3,-2.25){$\RSTS_n \cap \calH(\alpha) = \emptyset$};

%\fill [gray!60] (0,-4) rectangle (0+0.5,-4+0.5);
\end{tikzpicture}
%\caption{Relationship among $m(\alpha), M_1(\alpha),$ and %$M_2(\alpha)$.}
\label{fig:mconstants}
\end{figure}

\begin{enumerate}
\item There are no STSs in $\calH(\alpha)$ with fewer than $2g+s-2$ squares. 
\item $n=2g+s-2$ if and only if $\Hol(O) = \Prim$ for all $O \in \calH(\alpha) \cap \STS_n$
\item $2g+s-2 < n \leq 4g+2s-5$ implies that all $O \in \calH(\alpha) \cap \STS_n$ are visibility tori; in fact $\Prim \subsetneq \Hol(O)$.
\item $n=4g+2s-4$ implies there exists $O \in \calH(\alpha) \cap \STS_n$ that is not a visibility torus.
\item There exists $N(\alpha)$ such that if $n > N(\alpha)$ then all $O \in \calH(\alpha) \cap \STS_n$ are non-visibility tori, i.e. for all such STSs, there exists $v \in \Prim - \Hol(O)$.
\end{enumerate}

\end{restatable*}

A key result that aids in proving this theorem is the following result by B. Dozier  regarding the number of saddle connections of length at most $r$ in a translation surface $X$.

\begin{theorem}[Dozier \cite{dozier}]\label{thm:dozier} Given a stratum $\calH(\alpha)$, there exists a constant $c$ such that for any $X \in  \calH(\alpha)$ of unit area, and any interval $I \subset S^1$, there exists a constant $R_0(X)$ such that for all $R > R_0(X)$,
$$N(X;R; I) \leq c \cdot |I| \cdot R^2$$ where $N(X;R;I)$ denotes the number of saddle connections in $X$ of length at most $R$ and with direction in $I \subset S^1$. 
\end{theorem}

Dozier's result follows that of Masur \cite{masur2, masur3}, who showed that for a fixed translation surface $X$ there are quadratic lower and upper bounds on the growth rate of $N(X;R;I)$ depending on $X$ and of Eskin-Masur \cite{eskinmasur} who showed that for almost every surface in a fixed stratum, there exists an exact quadratic asymptotic for the growth rate of saddle connections, with the constant solely depending on the stratum. Since STSs form a measure zero set in the space of translation surfaces, the Eskin-Masur result is not enough for our purposes, and we must resort to Dozier's theorem.

Apart from the holonomy set, there are other ways that surfaces can be torus-like, one of which is by having $SL_2(\ZZ)$ as their Veech groups.
We investigate this property, proving for example the following simple statement.
\begin{restatable*}[Symmetry]{theorem}{symmetrytori}
\label{symmetrytori}
There are no genus-two reduced STSs with Veech group $SL_2(\ZZ)$.
\end{restatable*}
\emph{Reduced} STSs are those that do not cover a torus branched over a single point, other than $\TT$.
Finally, the torus has short curves in both the vertical and horizontal direction. We close by 
 investigating the asymptotics of STSs in the stratum $\calH(2)$ with a unit (length-one) holonomy vector which we call a \emph{unit saddle}.
If we fix a stratum, one can easily construct an infinite family of STSs with a unit saddle. 
However, in that fixed stratum,  the frequency of such surfaces  may still  tend to zero with the number of squares.

\begin{restatable*}[Unit saddles] {theorem}{total}
\label{total}
The number of $n$-square surfaces in $\calH(2)$ is $\Theta(n^3)$, 
while the number of $n$-square surfaces in $\calH(2)$ with a unit saddle is $\Theta(n^2)$. 
\end{restatable*}

One consequence of this result is that asymptotically almost surely, a random STS in $\calH(2)$ has no unit saddle. Since $\calH(2)$ has only one cone point, 
square-tiled surfaces with unit saddles are precisely those whose systoles (shortest simple closed curves) are of length 1 in that stratum. 
This fits into a greater body of work on the study of systoles of translation surfaces, including recent work by Boissy and Geninska \cite{boissygeninska} studying the local maxima of the shortest holonomy vector function on strata with $g\geq 3$, and by Judge and Parlier \cite{judgeparlier} studying the maximum number of systoles in genus 2.

Our paper is organized the following way. Section \ref{sec:background} contains relevant background information on translation surfaces and square-tiled surfaces. Section \ref{sec:faketori} contains the various notions of `fake tori'---STSs that are similar to $\TT$ in some aspect---including  proofs of  implications and non-implications between the different types. Section \ref{sec:finiteness} contains proofs Theorems \ref{faketori} and \ref{symmetrytori}. Section \ref{sec:unitloops} is devoted to proving Theorem \ref{total}. 

\textbf{Acknowledgements.} We are extremely grateful to Moon Duchin for organizing the Polygonal Billiards Research Cluster where this project originated, for suggesting this project, and for her constant support and valuable feedback. We are also grateful for the feedback and conversations with Samuel Leli\`evre. We are grateful for Madeleine Elyze and Luis Kumanduri for helping during the initial explorations. We thank  Vincent Delecroix and Justin Lanier for the prompt and helpful feedback on our first version. We also express our many thanks to the participants and visitors of the Billiards research cluster for numerous helpful discussions.

\textbf{Funding} This work began in the Polygonal Billiards Research Cluster held at Tufts University in Summer 2017 and was supported by the National Science Foundation under grant [DMSCAREER-1255442]. The second author was also supported by the National Science Foundation Graduate Research Fellowship under Grant No. 1122374.

%%%%%
\section{Background}
\label{sec:background}

\subsection{Translation surfaces and their moduli spaces} 

A \textbf{translation surface} can be defined geometrically as a collection of polygons in the plane with sides identified in parallel opposite pairs by translation, up to equivalence by cut and paste operations. They can also be defined as pairs $(X, \omega)$ where $\omega$ is a holomorphic one-form on the Riemann surface $X$. Translation surfaces are locally flat except at finitely many \textbf{cone points} or \textbf{singularities} where they have a cone angle of $2 \pi n$ for some integer $n \geq 2$. A cone point of angle $2 \pi n$ will correspond to a zero of the one-form $\omega$ of order $n-1$. 

A Gauss-Bonnet type theorem holds for translation surfaces: if $g$ is the genus of the surface and $\alpha_1, \ldots, \alpha_s$ are the degrees of the zeros of the one-form, then the surface must satisfy the relation $\sum_{i=1}^s \alpha_i = 2g-2.$  Genus $g$ translation surfaces then fall into finitely many \textbf{strata} $\mathcal{H}(\alpha)$ where $\alpha = (\alpha_1, \ldots, \alpha_s)$ describes the orders of the zeros of the one-form, $s$ zeros of orders $\alpha_1, \ldots, \alpha_s$. For example, Figure \ref{fig:STS} shows a translation surface in $\calH(1,1)$. On a given translation surface, a {\bf saddle connection} is a straight-line geodesic segment whose endpoints are cone points but with no cone points on the interior.  The {\bf holonomy} of a saddle connection is the corresponding Euclidean displacement vector. 

There is a natural $SL_2 \RR$ action on a stratum of translation surfaces. If we think of a translation surface as a collection of polygons in $\RR^2$ with side identifications, then we can act on this surface by a matrix $M \in SL_2 \RR$ by acting on the whole of $\RR^2$ by $M$. The new surface is the linearly transformed collection of polygons with the same side identifications. The stabilizer of this action as a subgroup of $SL_2 \RR$ is defined as the \textbf{Veech group} of the surface which will be denoted by $SL(X, \omega)$. % When $\omega$ is obvious from context, we will omit it from our notation. 

\subsection{Introduction to square-tiled surfaces} 

A \textbf{square-tiled surface} is a translation surface that is a covering of the square torus, branched over only one point. Geometrically, square-tiled surfaces can be constructed from finitely many unit squares, with sides glued in parallel opposite pairs. There is a natural $SL_2 \ZZ$ action on square-tiled surfaces that preserves both the stratum that the surface is in and the number of squares that is has. 

Sometimes, we wish only to consider square-tiled surfaces that do not cover other square-tiled surfaces. A square-tiled surface is called \textbf{reduced} if the only genus one square-tiled surface that it covers (branched over a single point) is the torus with only one square. Equivalently, a square-tiled surface is reduced if the $\ZZ$ lattice generated by its holonomy vectors (sometimes referred to as the \textbf{lattice of periods} is $\ZZ^2$). A square-tiled surface is called \textbf{primitive} if the only square-tiled surfaces that it covers are itself and the one square torus. Note that primitive STSs are reduced as well.

Recall that $\STS_n$ denotes the set of connected square-tiled surfaces built out of $n$ squares. Similarly, 
\begin{itemize}
% \item $\Permpair_n= S_n\times S_n / S_n$, where the quotient is by simultaneous conjugation. This is the set of (not necessarily connected) square-tiled surfaces with $n$ squares.
\item $\RSTS_n$ will be the set of reduced square-tiled surfaces with $n$ squares.
\item $\PSTS_n$ will be the set of primitive square-tiled surfaces with $n$ squares. 
\end{itemize} 

Using the {\sf Surface Dynamics} Sage package (\cite{surfacedyn}), we obtain counts of these categories of STSs for up to 10 squares. All of the surfaces enumerated below have genus at most 5.

%$$
%\begin{array}{cccc}
%n & |\STS_n| & |\RSTS_n| & |\PSTS_n| \\
%\hline
%1  & 1 & 1 & 1 \\
%2  & 3 & 0 & 0 \\
%3  & 7 & 3 & 3\\
%4  & 26 & 19 & 13\\
%5  & 97 & 91 & 91 \\
%6  & 624 & 603 & 500\\
%7  & 4163 & 4155 & 4155\\
%8  & 34470 & 34398 & 33190 \\
%9  & 314493 & 314468 & 313474\\
%10 & 3202839 & 3202548 & 3176532\\
%\hline
%\end{array}
%$$

$$
\begin{array}{c|cccccccccc}
n &  1 & 2 & 3 & 4 & 5 & 6 & 7 & 8 & 9 & 10\\
\hline
|\STS_n| & 1 & 3 & 7 & 26 & 97 & 624 & 4163 & 34470 & 314493 & 3202839 \\
\hline
 |\RSTS_n| & 1 & 0 & 3 & 19 & 91 & 603 & 4155 & 34398 & 314468 & 3202548 \\
 \hline
 |\PSTS_n| & 1 & 0 & 3 & 13 & 91 & 500 & 4155 & 33190 & 313474 & 3176532\\
 \hline
\end{array}
$$

\vspace{5mm}
% PermPair numbers: 1, 4, 11, 43, 161, 901, 5579, 43206, 378360, 3742738

We notice that the number of square-tiled surfaces with $n$ squares grows quite rapidly with $n$, which leads to challenges when trying to experimentally determine various statistics about square-tiled surfaces. We also see that at least for square-tiled surfaces with few squares, the counts are quite close. While we do not know of a proof of this, the experimental evidence in the table above suggests that the proportion of STSs with $n$ squares that are either reduced or primitive approaches $1$ as $n$ approaches infinity. 

%\begin{proposition}
%The counts of square-tiled surfaces and reduced square-tiled surfaces in a given stratum are related by
%$$|\STS_n \cap \calH(\alpha)|=\sum_{k | n} |\RSTS_{n/k} \cap \calH(\alpha)| \cdot \sigma_1(k).$$
%\end{proposition}
%
%\begin{proof}
%Let us establish a bijection between $\STS_n \cap \calH(\alpha)$ and the union over $k|n$ of the sets $(\RSTS_{n/k} \cap \calH(\alpha)) \times \Lambda_k$, where $\Lambda_k$ is the set of index $k$ sublattices of $\mathbb{Z}^2$. 
%
%Let $X \in \calH(\alpha) \cap \STS_n$, and let $\Lambda(X) = (a,b) \mathbb{Z} \oplus (c,d) \mathbb{Z}$ be its lattice of periods. Then, to $X$, we associate the element $$\left(\begin{bmatrix} a & c \\ b & d \end{bmatrix}^{-1} X, \Lambda(X)\right) \in (\RSTS_{n/k} \cap \calH(\alpha)) \times \Lambda_k,$$ where $k$ is the index of $\Lambda(X)$ in $\mathbb{Z}^2$. Conversely, if we have a pair $(Y, \Lambda) = (Y, (a,b) \mathbb{Z} \oplus (c,d) \mathbb{Z})$ in $(\RSTS_{n/k} \cap \calH(\alpha)) \times \Lambda_k$, then the reverse map associates to this pair the surface 
%$$\begin{bmatrix} a & c \\ b & d \end{bmatrix} Y \in \STS_n \cap \calH(\alpha).$$
%
%This establishes the bijection. To finish the proof, we recall that there are $\sigma_1(k)$ sublattices of $\ZZ^2$ of index $k$, and so $|\Lambda_k| = \sigma_1(k)$. By taking the sizes of both sides of the bijection 
%$$\STS_n \cap \calH(\alpha) \cong \left(\bigcup_{k | n} \RSTS_{n/k} \cap \calH(\alpha) \right) \times \Lambda_k,$$
%we get the statement of the proposition. 
%\end{proof} 

%%%%%
\subsection{Cylinder Decomposition}

Given a translation surface, a \textbf{cylinder} is a maximal collection of parallel closed geodesics.  It is well known that in any fixed rational direction, the non-singular straight line trajectories on a square-tiled surface are periodic. Hence, an STS decomposes as a union a finitely many cylinders in this direction. This is sometimes called a \textbf{cylinder decomposition}. In particular, every STS has a cylinder decomposition in the horizontal direction.

We will be using the possible horizontal cylinder decomposition of STSs in genus two for our arguments later. From Zorich \cite{zorich}, we know that any STS in $\calH(2)$ has either one or two cylinders in the horizontal direction. Similarly, we know that in $\calH(1,1)$, there are four possible types of cylinder decompositions. For details on the latter, refer to the Appendix of \cite{shrestha}.

Knowing the finite list of possible cylinder decompositions in each of the genus two strata allows us to parametrize STSs in them. For such a parametrization we refer the reader to Appendix  \ref{appendix:parametrization}.

\section{Notions of fake tori}
\label{sec:faketori}

%Recall that reduced STSs are those that do not cover any genus-one STS branched over a single point, except for $\TT$ itself. 
%The torus is clearly the simplest reduced STS and has many nice symmetry properties that are sometimes exhibited by other STSs as well. One example of extra symmetry in STSs that has been studied before is that of a {\em characteristic} STSs. 
%A surface $O$ is called a characteristic STS if the fundamental group of $O^*$, the surface obtained from $O$ by puncturing the surface at all of the corners of the squares, is a characteristic subgroup of $F_2$, the free group on two generators. These surfaces are known to have $SL_2\ZZ$ Veech groups (due to work in \cite{schmithusen}). The first non-trivial example is 
%the EW surface described above (see Figure~\ref{fig:pig}). There is also a method due to Herrlich \cite{herrlich} of constructing characteristic covers of any given STS and in this way, one can construct families of characteristic STSs as well. Note, however, the characteristic STSs so produced are not necessarily reduced. 

In this section, we explore various ways in which a square-tiled surface can be a \emph{fake torus}: that is, like the one-square torus in some way. The simplicity of the square torus contributes to it having many nice properties like having the full $SL_2 \ZZ$ as its Veech group and the full set of primitive vectors in $\ZZ^2$ as its set of holonomy vectors. Our goal for this section is to determine the implications and non-implications between these properties through proofs, experimental evidence, and examples.

\subsection{Definitions}

In the following, let $\Hol(O)$ denote the set of holonomy vectors of $O\in \STS_n$ and $\Prim:=SL_2\ZZ \cdot (1,0)$, the set of primitive vectors in $\ZZ^2$ (those with relatively prime coordinates).
If $O$ is a square-tiled surface and $\TT$ is the one square torus with one marked point, then there is a natural map $f : O \rightarrow \TT$. We will denote $\TT$ punctured at its marked point as $\TT^*$. Then, $f^{-1} (\TT^*)$ is the square-tiled surface $O$ punctured at all of its vertices. Recall that a subgroup $H \subseteq G$ is called a characteristic subgroup of $G$ if every automorphism of $G$ fixes $H$.

\begin{definition}\
\begin{enumerate}
\item A \textbf{symmetry torus} is a reduced STS  with $\Veech(O)=\Veech(\TT)=SL_2\ZZ$. 
\item  A \textbf{holonomy torus} is an STS with $\Hol(O)=\Hol(\TT)=\Prim$.
\item  $O\in \STS_n$ is a {\bf visibility torus} if $\Prim \subseteq \Hol(O)$; otherwise, it is {\bf non-visibility}.
\item $O$ is a \textbf{characteristic STS} if $\pi_1(f^{-1} (\TT^*)) \subseteq \pi_1 (\TT^*) \cong F_2$ is a characteristic subgroup of $F_2$. Similarly,  $O$ is a \textbf{normal STS} if $\pi_1(f^{-1} (\TT^*)) \subseteq \pi_1 (\TT^*) \cong F_2$ is a normal subgroup of $F_2$.
\end{enumerate}
\end{definition}

Alternatively, if we let $\Gamma = \pi_1(f^{-1}(\TT^*)) \subset \pi_1(\TT^*) \cong F_2$, then $O$ is a symmetry torus if for every automorphism $\varphi$ of $F_2$, $\varphi(\Gamma) = g \Gamma g^{-1}$ for some $g \in F_2$. 

\begin{remark} We will work only with reduced square-tiled surfaces in this section. Being a holonomy torus or visibility torus implies that $\Prim \subseteq \Hol(O) $, and so its lattice of periods is $\ZZ^2$, and the STS is reduced. Symmetry tori are reduced by definition. Characteristic STSs of genus greater than 1 are necessarily reduced, as an easy consequence of Lemma $\ref{lem:normalholonomy}$.
\end{remark}

We can summarize the known relationships between several types of STSs defined in this section in Figure \ref{fig:implications}. Subsequently, we will give citations or proofs for the implications and non-implications. 
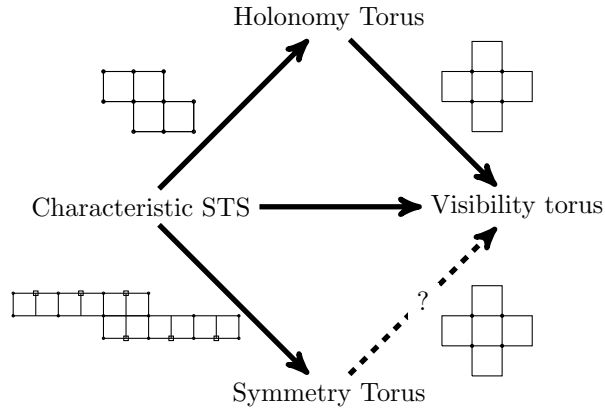
\begin{figure}[hb]
\begin{tikzpicture}[scale = 0.5]
%\node at (0,0) [below] {Symmetry Torus};
%\node at (1,1) [right] {visibility torus};
%\node at (-1,1) [left] {Characteristic STS};
%\node at (0,2) [above] {Holonomy Torus};
%
%\draw [->, line width=3] (-1,1) -- (0,2);
%\draw [->, line width=3] (-1,1) -- (0,0);
%\draw [->, line width=3] (0,2) -- (1,1);
%\draw [->, dashed, line width=3] (0,0) -- (1,1);
%\useasboundingbox (-1,-2) rectangle (8,2);
%\SetVertexSimple
%\SetGraphUnit{5}
\tikzset{VertexStyle/.style = {shape = rectangle}}
\Vertex{Characteristic STS}
%\tikzset{VertexStyle/.style={shape=coordinate}}
\SOEA[unit=5](Characteristic STS){Symmetry Torus}
\NOEA[unit=5](Characteristic STS){Holonomy Torus}
\EA[unit=10](Characteristic STS){Visibility torus}
%\Loop[dist=3cm,dir = SO, style={thick,->}](B)
%\Loop[dist=5cm,dir=NO,style={thick,->}](A)
%\Loop[dist=3cm,dir=SO,style={thick,<-}](A)
\Edge[lw = {2pt}, style={->, >=stealth'}](Characteristic STS)(Symmetry Torus)
% \Edge[lw = {2pt}, label={|}, style={->, >=stealth'}](Symmetry Torus)(Characteristic STS)
\Edge[lw = {2pt}, style={->, >=stealth'}](Characteristic STS)(Holonomy Torus)
\Edge[lw = {2pt}, style={->, >=stealth'}](Characteristic STS)(Visibility torus)
\Edge[lw = {2pt}, style={->, >=stealth'}](Holonomy Torus)(Visibility torus)
\Edge[lw = {2pt}, label ={?}, style={dashed, ->, >=stealth'}](Symmetry Torus)(Visibility torus)
% \Edge[lw = {2pt}, label ={|}, style={ ->, >=stealth'}](Symmetry Torus)(Holonomy Torus)
%\Edge[style={-, bend right = 45}](A)(C)
%\tikzset{EdgeStyle/.style = {-,bend left=90}}
%\Edge[lw = 10pt,  style={thick, -, bend right = 90}](C)(B)
%\Loop[dist=10cm,dir=SO,style={thick,->}](A)
\begin{scope}[xshift=8cm,yshift=2cm,scale=.8]
\draw (1,0) -- (2,0) -- (2,3) -- (1,3) -- (1,0); 
\draw (0,1) -- (3,1) -- (3,2) -- (0,2) -- (0,1); 
\draw [fill] (1,1) circle [radius=0.05];
\draw [fill] (2,1) circle [radius=0.05];
\draw [fill] (1,2) circle [radius=0.05];
\draw [fill] (2,2) circle [radius=0.05];
\end{scope} 
\begin{scope}[xshift=-1cm,yshift=2cm,scale=.8]
\draw (0,1) -- (1,1) -- (1,0) -- (3,0) -- (3,1) -- (2,1) -- (2,2) -- (0,2) -- (0,1);
\draw (1,1) -- (1,2); 
\draw (2,0) -- (2,1); 
\draw (1,1) -- (2,1); 

\draw [fill] (0,1) circle [radius=0.05];
\draw [fill] (0,2) circle [radius=0.05];
\draw [fill] (2,1) circle [radius=0.05];
\draw [fill] (2,0) circle [radius=0.05];
\draw [fill] (2,2) circle [radius=0.05];
\draw (1,0) circle [radius= 0.05];
\draw (1,1) circle [radius= 0.05];
\draw (1,2) circle [radius= 0.05];
\draw (3,0) circle [radius= 0.05];
\draw (3,1) circle [radius= 0.05];
\end{scope} 

%\begin{scope}[xshift=-1cm,yshift=-3cm,scale=.6,rotate=-45]
\begin{scope}[xshift=-1cm,yshift=-3.5cm,scale=.6]
\draw (0,0) -- (6,0) -- (6,1) -- (0,1) -- (0,0); 
\draw (-4,1) -- (2,1) -- (2,2) -- (-4, 2) -- (-4, 1); 
\draw (-3,1) -- (-3, 2);
\draw (-2,1) -- (-2, 2);
\draw (-1,1) -- (-1, 2);
\draw (0,1) -- (0, 2);
\draw (1,0) -- (1, 2);
\draw (2,0) -- (2, 1);
\draw (3,0) -- (3, 1);
\draw (4,0) -- (4, 1);
\draw (5,0) -- (5, 1);

\draw [fill] (-4,1) circle [radius=0.05];
\draw [fill] (-2,1) circle [radius=0.05];
\draw [fill] (0,1) circle [radius=0.05];
\draw [fill] (2,1) circle [radius=0.05];
\draw [fill] (4,1) circle [radius=0.05];
\draw [fill] (6,1) circle [radius=0.05];

\draw (-4,2) circle [radius=0.05];
\draw (-2,2) circle [radius=0.05];
\draw (0,2) circle [radius=0.05];
\draw (2,2) circle [radius=0.05];
\draw (0,0) circle [radius=0.05];
\draw (2,0) circle [radius=0.05];
\draw (4,0) circle [radius=0.05];
\draw (6,0) circle [radius=0.05];

\draw (-3.1,1.9) rectangle (-2.9, 2.1);
\draw (-1.1,1.9) rectangle (-0.9, 2.1);
\draw (0.9,1.9) rectangle (1.1, 2.1);
\draw (0.9,-0.1) rectangle (1.1, 0.1);
\draw (2.9,-0.1) rectangle (3.1, 0.1);
\draw (4.9,-0.1) rectangle (5.1, 0.1);

\end{scope}

\begin{scope}[xshift=8cm,yshift=-4.5cm,scale=.8]
\draw (1,0) -- (2,0) -- (2,3) -- (1,3) -- (1,0); 
\draw (0,1) -- (3,1) -- (3,2) -- (0,2) -- (0,1); 
\draw [fill] (1,1) circle [radius=0.05];
\draw [fill] (2,1) circle [radius=0.05];
\draw [fill] (1,2) circle [radius=0.05];
\draw [fill] (2,2) circle [radius=0.05];
\end{scope} 
\begin{scope}[xshift=-1cm,yshift=2cm,scale=.8]
\draw (0,1) -- (1,1) -- (1,0) -- (3,0) -- (3,1) -- (2,1) -- (2,2) -- (0,2) -- (0,1);
\draw (1,1) -- (1,2); 
\draw (2,0) -- (2,1); 
\draw (1,1) -- (2,1); 

\draw [fill] (0,1) circle [radius=0.05];
\draw [fill] (0,2) circle [radius=0.05];
\draw [fill] (2,1) circle [radius=0.05];
\draw [fill] (2,0) circle [radius=0.05];
\draw [fill] (2,2) circle [radius=0.05];
\draw (1,0) circle [radius= 0.05];
\draw (1,1) circle [radius= 0.05];
\draw (1,2) circle [radius= 0.05];
\draw (3,0) circle [radius= 0.05];
\draw (3,1) circle [radius= 0.05];
\end{scope} 

\end{tikzpicture}
\caption{Summary of relationships among reduced surfaces, as proved in this section. 
The surfaces depicted here show counterexamples for the reverse implications (redrawn in full, below). The dotted line is a conjectured implication.}
\label{fig:implications}
\end{figure}

The following propositions are alternate characterizations of holonomy tori and visibility tori which will be useful later.

\begin{proposition}\label{prop:althol} $O$ is a holonomy torus if and only if every corner of every square is a singularity.
\end{proposition}
\begin{proof} If $O$ is not a holonomy torus, then it has a non-primitive vector as a saddle connection. This saddle connection then must go through a corner of a square that is not a singularity. Conversely, if there exists a corner of a square that is not a singularity, then extending a line segment through this point in any rational slope direction yields a non-primitive vector as a saddle connection and hence the surface is not a holonomy torus.
\end{proof} 
%\begin{proof}
%$(\Rightarrow)$:
%Let $O$ be a holonomy torus with a cone point, $p$. Assume $O$ has a corner of a square that is not a singularity, $q$. Then there exists a straight locally geodesic segment from some singularity $p$ to $q$ without any singularities within, with a rational direction $\vec{v}$. Extending this segment in the same direction, we must eventually hit a singularity, $r$ ($r=p$ is possible), since given any rational direction, all STSs decompose into cylinders in that direction. Then, the holonomy vector of the saddle connection so formed from $p$ to $r$ will not be in $\Prim$.
%
%$(\Leftarrow)$:
%Let every corner of every square be a singularity. Then every saddle connection would have primitive holonomy vector, otherwise the saddle connection would need to go through a non-singular corner. So, $\Hol(O) \subset \Prim$. Moreover, let $v \in \Prim$. Then, starting at a singularity, consider a straight line trajectory in direction $v$. As $v$ is rational, the surface decomposes into cylinders and it must either return to the singular point it started from, or pass through another singularity. However, it cannot pass through a non-singular corner point, as none such points exist in $O$. Therefore, the saddle connection formed by this trajectory must have holonomy $v$ so that $v \in \Hol(O)$.   
%\end{proof}

\begin{proposition}\label{prop:altunobs}$O$ is a visibility torus if and only if every surface in the $SL_2 \ZZ$ orbit of $O$ has a unit length horizontal saddle connection.
\end{proposition}
\begin{proof}
$(\Rightarrow)$:
Let $O$ be a visibility torus, and let $O' \in SL_2\ZZ \cdot O$ such that $O' = M \cdot O$ for some $M = \begin{bmatrix}p & r \\ q & s \end{bmatrix} \in SL_2\ZZ$. Since $O$ is visibility, $(s, -q) \in \Hol(O)$. But, $\Hol(O') = M \cdot \Hol(O)$ which implies $M \cdot (s, -q) = (ps-qr,0) = (1,0) \in \Hol(O')$ and hence $O'$ has a unit length horizontal saddle connection.\\
$(\Leftarrow)$:
Let every surface in the orbit of $O$ have a unit length horizontal saddle connection. 
Let $(p,q) \in \Prim$ be arbitrary, so that $ps - rq = 1$ for some $s, q \in \NN$. Consider a matrix $M = \begin{bmatrix}p & r \\ q & s \end{bmatrix} $ in $SL_2\ZZ$. Then $O' := M^{-1} \cdot O$ is in the $SL_2\ZZ$ orbit of $O$, and by assumption, $(1,0) \in \Hol(O')$. As $O = M \cdot O'$, we have that $\Hol(O) = M \cdot \Hol(O')$ which implies $M \cdot (1,0) = (p,q) \in \Hol(O)$. \end{proof}

We remark here that one can easily replace the `horizontal' with `vertical' in Proposition \ref{prop:altunobs}.

\subsection{Implications} \label{sec:implications}

We start with a proposition, which follows from the work of G. Schmith\"usen \cite{schmithusen}: 
\begin{proposition}[Schmith\"usen]
\label{prop:charsym} $O$ is a characteristic STS if and only if  $O$ is a normal symmetry torus.
\end{proposition}

To work towards other implications, we will start with a lemma relating holonomy tori with normal STSs.

\begin{lemma} 
\label{lem:normalholonomy}
If $O$ is a normal STS, then $O$ is a holonomy torus. 
\end{lemma} 

\begin{proof}
We may assume that $O$ has a cone point; otherwise, if it is reduced, then it is the standard torus.
\InsertBoxR{2}{\begin{minipage}{0.25\linewidth}\centering
\begin{tikzpicture}[scale=0.8]
\draw (0,0) -- (4,0) -- (4,1) -- (0,1) -- (0,0); 
\draw (0,0) -- (2,0) -- (2,2) -- (0,2) -- (0,0); 
\draw (1,0) -- (1,2); 
\draw (3,0) -- (3,1); 
\node at (0.5, -0.3) {$1$};
\node at (1.5, -0.3) {$2$};
\node at (0.5, 2.3) {$2$};
\node at (1.5, 2.3) {$1$};
\draw [fill] (0,0) circle [radius=0.05];
\draw [fill] (2,0) circle [radius=0.05];
\draw [fill] (2,1) circle [radius=0.05];
\draw [fill] (0,1) circle [radius=0.05];
\draw [fill] (1,2) circle [radius=0.05];
\draw[->] (0.5, 0.5) -- (1.5, 0.5);
\draw[->] (1.5, 0.5) -- (1.5, 1.5);
\draw (1.5, 1.5) -- (1.5, 2);
\draw[->] (0.5, 0) -- (0.5, 0.5);
\end{tikzpicture}  
\end{minipage}}[2]
If $O$ is a normal STS, then let $O^*$ be the STS punctured at every vertex. Fixing a basepoint in $O^*$ in a given square $S$ of $O^*$, elements in $\pi_1(O^*)$ are paths of squares from square $S$ to itself. These paths are represented as elements of $F_2 \cong \langle U, R\rangle$, where $U$ is an up step, and $R$ is a right step. For example, the path in the surface below starting in the bottom left square is represented by  $RU^2$.

Since $\pi_1(O^*)$ is normal in $F_2$, it is invariant under conjugation. Therefore a word $w$ representing a path of up/down/left/right moves forms a closed loop based at $S$ if and only if it forms a closed loop based at every other square. 

If $O$ has genus $\geq 2$, then we can choose $S$ to be a square that has a singularity in its top right corner. Then, the word $RUR^{-1}U^{-1}$ is a path of angle $2\pi$ around the singularity, counterclockwise, and begins at $S$ and ends at a square that is not $S$. This implies that $RUR^{-1} U^{-1}$ is not a closed loop based at any other square, so every square has a singularity in its top right corner. Therefore, $O$ is a holonomy torus. 
\end{proof}

It follows immediately that if $O$ is a characteristic STS, then $O$ is a holonomy torus.
Since holonomy tori are a subset of visibility tori by definition, it is also immediate  that characteristic STSs are a subset of visibility tori.

\subsection{Non-implications}\label{sec:nonimplications} We now present some examples showing non-implications between the different notions of fake tori.
\begin{wrapfigure}{R}{0.2\textwidth}
\begin{center}
\vspace{-20pt}
\begin{tikzpicture}[scale=0.6]
\draw (1,0) -- (2,0) -- (2,3) -- (1,3) -- (1,0); 
\draw (0,1) -- (3,1) -- (3,2) -- (0,2) -- (0,1); 
\draw [fill] (1,1) circle [radius=0.05];
\draw [fill] (2,1) circle [radius=0.05];
\draw [fill] (1,2) circle [radius=0.05];
\draw [fill] (2,2) circle [radius=0.05];
\end{tikzpicture} 
\vspace{10pt}

\begin{tikzpicture}[scale=0.6] 
\draw (0,1) -- (1,1) -- (1,0) -- (3,0) -- (3,1) -- (2,1) -- (2,2) -- (0,2) -- (0,1);
\draw (1,1) -- (1,2); 
\draw (2,0) -- (2,1); 
\draw (1,1) -- (2,1); 

\draw [fill] (0,1) circle [radius=0.05];
\draw [fill] (0,2) circle [radius=0.05];
\draw [fill] (2,1) circle [radius=0.05];
\draw [fill] (2,0) circle [radius=0.05];
\draw [fill] (2,2) circle [radius=0.05];
\draw (1,0) circle [radius= 0.05];
\draw (1,1) circle [radius= 0.05];
\draw (1,2) circle [radius= 0.05];
\draw (3,0) circle [radius= 0.05];
\draw (3,1) circle [radius= 0.05];
\end{tikzpicture} 
\end{center}
\vspace{-25pt}
\end{wrapfigure}

\textbf{Example}. The Swiss Cross shown on the right is a surface in $\calH(2)$ that is visibility, but is not characteristic, nor a symmetry torus, nor a holonomy torus.  The fact that it is visibility can be proven using Proposition \ref{prop:altunobs}. 

\textbf{Example}. The four-square STS on the right is a holonomy torus that is not a symmetry torus because the element $\begin{bmatrix} 1 & 1 \\ 0 & 1 \end{bmatrix}$ does not belong to its Veech group. In fact, its Veech group, computed via the {\sf Surface Dynamics} package \cite{surfacedyn}, has index 6 in $SL_2\ZZ$. 
Thus it is also not characteristic.

\begin{example}

The two surfaces in Figure \ref{fig:2sym} are examples of symmetry tori that are not holonomy tori, and therefore not characteristic. 
The surfaces have 12 and 16 squares, respectively, but both have genus four and belong to the same stratum, $\mathcal H(2,2,2)$. The surface with 12 squares was first discovered by Forni and Matheus \cite{fornimatheus}, and is often called the \emph{Ornithorynque} for its unusual properties. The \emph{Ornithorynque}  is an example of a general construction of square-tiled cyclic covers, detailed in \cite{fornimatheuszorich}.
%The STS defined by $\sigma=(1,2,3,4,5,6)(7,8,9,10,11,12)$, $\tau=(1,7,5,9,3,11)(2,8,4,12,6,10)$ (shown below) 
These two examples are both visibility, meaning that all relatively prime vectors appear in the holonomy set, but they are not holonomy tori because $(2,2)$ is seen to be a holonomy vector in each case.  We have verified experimentally that these are the only examples of symmetry but not holonomy tori with $\le 16$ squares. It appears that the one with 16 squares is the first known example of a symmetry torus where the lengths of the cylinders are not all equal.

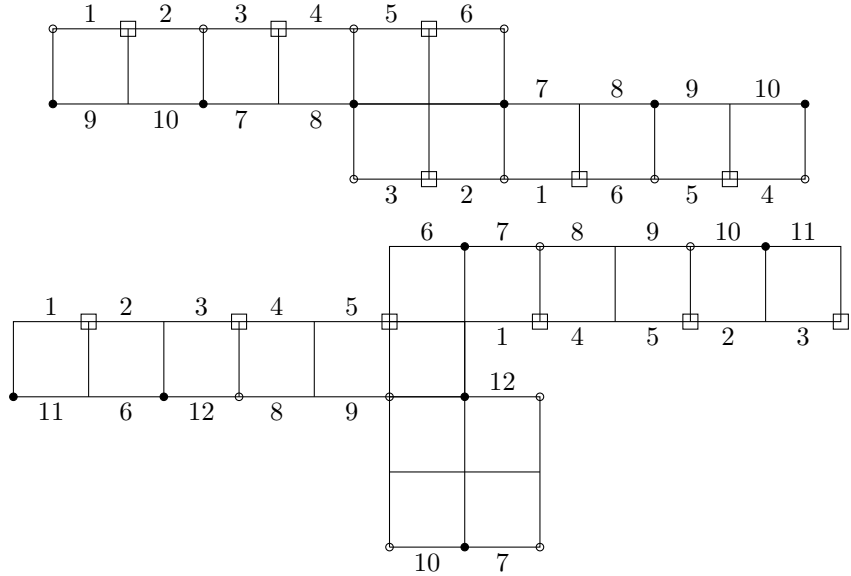
\begin{figure}[ht]
\begin{tikzpicture}
\draw (0,0) -- (6,0) -- (6,1) -- (0,1) -- (0,0); 
\draw (-4,1) -- (2,1) -- (2,2) -- (-4, 2) -- (-4, 1); 
\draw (-3,1) -- (-3, 2);
\draw (-2,1) -- (-2, 2);
\draw (-1,1) -- (-1, 2);
\draw (0,1) -- (0, 2);
\draw (1,0) -- (1, 2);
\draw (2,0) -- (2, 1);
\draw (3,0) -- (3, 1);
\draw (4,0) -- (4, 1);
\draw (5,0) -- (5, 1);
\node at (-3.5, 2.2) {$1$};
\node at (-2.5, 2.2) {$2$};
\node at (-1.5, 2.2) {$3$};
\node at (-0.5, 2.2) {$4$};
\node at (0.5, 2.2) {$5$};
\node at (1.5, 2.2) {$6$};
\node at (2.5, 1.2) {$7$};
\node at (3.5, 1.2) {$8$};
\node at (4.5, 1.2) {$9$};
\node at (5.5, 1.2) {$10$};

\node at (-3.5, 0.8) {$9$};
\node at (-2.5, 0.8) {$10$};
\node at (-1.5, 0.8) {$7$};
\node at (-0.5, 0.8) {$8$};
\node at (0.5, -0.2) {$3$};
\node at (1.5, -0.2) {$2$};
\node at (2.5, -0.2) {$1$};
\node at (3.5, -0.2) {$6$};
\node at (4.5, -0.2) {$5$};
\node at (5.5, -0.2) {$4$};

\draw [fill] (-4,1) circle [radius=0.05];
\draw [fill] (-2,1) circle [radius=0.05];
\draw [fill] (0,1) circle [radius=0.05];
\draw [fill] (2,1) circle [radius=0.05];
\draw [fill] (4,1) circle [radius=0.05];
\draw [fill] (6,1) circle [radius=0.05];

\draw (-4,2) circle [radius=0.05];
\draw (-2,2) circle [radius=0.05];
\draw (0,2) circle [radius=0.05];
\draw (2,2) circle [radius=0.05];
\draw (0,0) circle [radius=0.05];
\draw (2,0) circle [radius=0.05];
\draw (4,0) circle [radius=0.05];
\draw (6,0) circle [radius=0.05];

\draw (-3.1,1.9) rectangle (-2.9, 2.1);
\draw (-1.1,1.9) rectangle (-0.9, 2.1);
\draw (0.9,1.9) rectangle (1.1, 2.1);
\draw (0.9,-0.1) rectangle (1.1, 0.1);
\draw (2.9,-0.1) rectangle (3.1, 0.1);
\draw (4.9,-0.1) rectangle (5.1, 0.1);
\end{tikzpicture}

%has permutation pair 
%$$(\sigma,\tau)=\left( (1,2)(3,4)(5,6,7,8,9,10)(11,12,13,14,15,16), \ \ (1,3,5,11,7,15)(2,4,8,16,6,12)(9,13)(10,14) \right)$$ and 

\begin{tikzpicture}
\draw (-5,0) -- (1,0) -- (1,1) -- (-5,1) -- (-5, 0); 
\draw (0,1) -- (6,1) -- (6,2) -- (0,2) -- (0,1); 
\draw (0,-2) -- (2,-2) -- (2, 0) -- (0,0) -- (0,-2); 
\draw (1,-2) -- (1,2); 
\draw (-4,0) -- (-4, 1);
\draw (-3,0) -- (-3, 1);
\draw (-2,0) -- (-2, 1);
\draw (-1,0) -- (-1, 1);
\draw (0,0) -- (0, 1);
\draw (2,1) -- (2,2);
\draw (3,1) -- (3,2);
\draw (4,1) -- (4,2);
\draw (5,1) -- (5,2);
\draw (0,-1) -- (2,-1); 

\node at (-4.5, 1.2) {$1$};
\node at (-3.5, 1.2) {$2$};
\node at (-2.5, 1.2) {$3$};
\node at (-1.5, 1.2) {$4$};
\node at (-0.5, 1.2) {$5$};
\node at (0.5, 2.2) {$6$};
\node at (1.5, 2.2) {$7$};
\node at (2.5, 2.2) {$8$};
\node at (3.5, 2.2) {$9$};
\node at (4.5, 2.2) {$10$};
\node at (5.5, 2.2) {$11$};
\node at (1.5, 0.2) {$12$};

\node at (-4.5, -0.2) {$11$};
\node at (-3.5, -0.2) {$6$};
\node at (-2.5, -0.2) {$12$};
\node at (-1.5, -0.2) {$8$};
\node at (-0.5, -0.2) {$9$};
\node at (1.5, 0.8) {$1$};
\node at (2.5, 0.8) {$4$};
\node at (3.5, 0.8) {$5$};
\node at (4.5, 0.8) {$2$};
\node at (5.5, 0.8) {$3$};
\node at (0.5, -2.2) {$10$};
\node at (1.5, -2.2) {$7$};

\draw [fill] (1,2) circle [radius=0.05];
\draw [fill] (5,2) circle [radius=0.05];
\draw [fill] (1,0) circle [radius=0.05];
\draw [fill] (1,-2) circle [radius=0.05];
\draw [fill] (-5,0) circle [radius=0.05];
\draw [fill] (-3,0) circle [radius=0.05];

\draw (2,2) circle [radius=0.05];
\draw (4,2) circle [radius=0.05];
\draw (-2,0) circle [radius=0.05];
\draw (0,0) circle [radius=0.05];
\draw (2,0) circle [radius=0.05];
\draw (0,-2) circle [radius=0.05];
\draw (2,-2) circle [radius=0.05];

\draw (-4.1,0.9) rectangle (-3.9, 1.1);
\draw (-2.1,0.9) rectangle (-1.9, 1.1);
\draw (-0.1,0.9) rectangle (0.1, 1.1);
\draw (1.9,0.9) rectangle (2.1, 1.1);
\draw (3.9,0.9) rectangle (4.1, 1.1);
\draw (5.9,0.9) rectangle (6.1, 1.1);
\end{tikzpicture} 
\caption{Two symmetry tori that are not holonomy tori. The one on top is Forni-Matheus's \emph{Ornithorynque}.}
\label{fig:2sym}
\end{figure} 
\end{example}

\textbf{Missing Implication.} We note that to prove the missing implication (symmetry torus implies visibility torus) from Figure \ref{fig:implications}, it suffices to show that for any $O \in \RSTS$ there exists $v \in \Hol(O) \cap \Prim$. For then, as symmetry tori are reduced by definition, and have $SL_2\ZZ$ Veech group, $\Prim = SL_2\ZZ \cdot v \subset \Hol(O)$.

\section{Counting fake tori}\label{sec:finiteness}

In this section we will make precise the intuition that not many surfaces should have the simplicity of the torus by proving that there are, for some definitions of fake torus from the previous section, only finitely many fake tori in each stratum of translation surfaces.

We approach this problem first by examining some computational evidence. Using the {\sf Surface Dynamics} Sage library (\cite{surfacedyn}), we can report a complete inventory of all reduced square-tiled surfaces with up to 10 squares and genus at most 5. 

$$
\begin{array}{ccccccc}
n & |\RSTS_n| & \# \text{Symm} & \# \text{Hol} & \# \text{Non-visibility} &\# \text{Hol}/ |\RSTS_n| & \# \text{Non-visibility}/|\RSTS_n|\\
\hline
2& 0 & 0 & 0 & 0& -- & --\\
3&3&0&3& 0 & 1 & 0\\
4&19&0&10& 0 & 0.526 & 0\\
5&91&0&40& 0 & 0.440 & 0\\
6&603&0&254& 36 & 0.421 & 0.060 \\
7&4155&0&1620& 90 & 0.390 & 0.022\\
8&34398&1&13364& 348 & 0.389 & 0.010\\
9&314468&0&119892& 693 & 0.381 &  0.0022\\
10&3202548&0&1212334&7491 & 0.379 & 0.0023
\end{array}
$$

The library also contains 3717 $SL_2\ZZ$ orbits with more than 10 squares, representing many millions of distinct reduced surfaces.
For example, restricting to  genus two, the library is complete up to 55 squares. From the data gathered above, it appears as though symmetry tori are very rare among square-tiled surfaces, but the proportion of $n$-square STSs that are holonomy tori or visibility tori may not tend toward zero as $n$ goes to infinity. 
In all of the examples computed in the database, there are a total of 3 symmetry tori: the two shown in Figure \ref{fig:2sym} and the EW surface in $\STS_8$, shown in Figure \ref{fig:pig}.

We will show the existence of an upper bound, in terms of the stratum, on the number of squares a visibility torus can have in a given stratum. These results on the number of squares allow us to show that given any stratum, there are finitely many visibility tori and holonomy tori. Thus, two seemingly opposing behaviors occur: while the table suggests that the proportion of $n$-square STSs that are holonomy tori or visibility tori is large, there are actually only finitely many holonomy or visibility tori in each stratum. 

The table also suggests that symmetry tori are rare. We will confirm this in genus two by showing that there are no symmetry tori in genus two.

\subsection{Counting problems related to visibility properties}

We devote this section to proving Theorem \ref{faketori}.

We start by proving that, given a stratum, there exists a minimum number of squares that a STS in the stratum can have, and that for every $n$ above this minimum, there exists a one-cylinder STS with $n$ squares in the given stratum. The following proposition is also proven algebraically by Zmiaikou in \cite{zmiaikouthesis} (Theorem 3.2), but we include a constructive proof here. A related construction of one-cylinder translation surfaces in each stratum that are not necessarily square-tiled is given by Zorich in \cite{zorichjenkin}. 

\begin{proposition} 
\label{prop:1cyl} 
$\STS_n \cap \calH(\alpha_1, \dots, \alpha_s) \neq \emptyset$ if and only if $n \geq 2g-2+s$.   
\end{proposition}
\begin{proof}
$(\Rightarrow):$ Let $O \in \STS_n \cap \calH(\alpha_1, \dots, \alpha_s)$. Then, the total angle of the cone points is $2 \pi \sum_{i=1}^s (\alpha_i+1)$. Moreover, the total angle contributed by the vertices, both singular and non-singular, is $2\pi n$. So, $2 \pi n \geq 2 \pi \sum_{i=1}^s (\alpha_i+1) \implies n \geq 2g-2 + s$.

$(\Leftarrow):$
We will start by showing that given any positive integer $k$, we can construct a $1$-cylinder surface with $2k+1$ squares, belonging to $\calH(2k)$. Given any positive integers $p$ and $q$, we can also construct a surface with $2p + 2q$ squares, belonging to $\calH(2p-1,2q-1)$. These will then be the building blocks that we use to construct $1$-cylinder surfaces in other strata. 

Let us first construct a $1$-cylinder surface in $\calH(2k)$ with $2k+1$ squares. If we label the squares in the cylinder from left to right as $1,2,\ldots,2k+1$, then we can represent the vertical side identifications by a permutation denoting which squares are glued to each other, going upward along their vertical cylinders. The vertical permutation that we use to construct our desired surface in $\calH(2k)$ is $\tau = (1) (2,3) (4,5) \ldots (2k, 2k+1)$, as seen below: 

\begin{center}
\begin{tikzpicture}
\draw (0,0) -- (5,0);
\draw (0,1) -- (5,1);
\draw (6,0) -- (8,0);
\draw (6,1) -- (8,1);
\draw (0,0) -- (0,1);
\draw (1,0) -- (1,1);
\draw (2,0) -- (2,1);
\draw (3,0) -- (3,1);
\draw (4,0) -- (4,1);
\draw (5,0) -- (5,1);
\draw (6,0) -- (6,1);
\draw (7,0) -- (7,1);
\draw (8,0) -- (8,1);
\node at (0.5, 1.2) {$1$};
\node at (1.5, 1.2) {$2$};
\node at (2.5, 1.2) {$3$};
\node at (3.5, 1.2) {$4$};
\node at (4.5, 1.2) {$5$};
\node at (6.5, 1.2) {$2k$};
\node at (7.5, 1.2) {$2k+1$};

\node at (0.5, -0.2) {$1$};
\node at (1.5, -0.2) {$3$};
\node at (2.5, -0.2) {$2$};
\node at (3.5, -0.2) {$5$};
\node at (4.5, -0.2) {$4$};
\node at (6.5, -0.2) {$2k+1$};
\node at (7.5, -0.2) {$2k$};

\node at (5.5, 0.5) {$\cdots$};
\draw [fill] (0,0) circle [radius=0.05];
\draw [fill] (1,0) circle [radius=0.05];
\draw [fill] (2,0) circle [radius=0.05];
\draw [fill] (3,0) circle [radius=0.05];
\draw [fill] (4,0) circle [radius=0.05];
\draw [fill] (5,0) circle [radius=0.05];
\draw [fill] (6,0) circle [radius=0.05];
\draw [fill] (7,0) circle [radius=0.05];
\draw [fill] (8,0) circle [radius=0.05];
\draw [fill] (0,1) circle [radius=0.05];
\draw [fill] (1,1) circle [radius=0.05];
\draw [fill] (2,1) circle [radius=0.05];
\draw [fill] (3,1) circle [radius=0.05];
\draw [fill] (4,1) circle [radius=0.05];
\draw [fill] (5,1) circle [radius=0.05];
\draw [fill] (6,1) circle [radius=0.05];
\draw [fill] (7,1) circle [radius=0.05];
\draw [fill] (8,1) circle [radius=0.05];
\end{tikzpicture}
\end{center}

Now, let us construct a surface in $\calH(2p-1,2q-1)$ with $2p+2q$ squares. We will construct this by concatenating from left to right: a block of $2p-2$ squares with the permutation $(1,2)(3,4)\cdots (2p-3,2p-2)$, a block of three squares with the permutation $(2p-1, 2p+1)(2p)$, another $2q-2$ squares with the permutation $(2p+2, 2p+3)(2p+3, 2p+4)\cdots (2p+2q-2, 2p+2q-1)$, and then one block with the permutation $(2p+2q)$. 

The picture of this is as follows (where $t := p+q$), and by following the side identifications around as before, we get the following vertex identifications: 

\begin{center}
\begin{tikzpicture}

\draw [fill=lightgray] (0,0) rectangle (2,1);
\draw [fill=lightgray] (3,0) rectangle (5,1);
\draw [fill=lightgray] (8,0) rectangle (10,1);
\draw [fill=lightgray] (11,0) rectangle (13,1);

\draw (0,0) -- (2,0);
\draw (0,1) -- (2,1);
\draw (0,0) -- (0,1);
\draw (1,0) -- (1,1);
\draw (2,0) -- (2,1);

\node at (2.5, 0.5) {$\cdots$};

\draw (3,0) -- (10,0);
\draw (3,1) -- (10,1);
\draw (3,0) -- (3,1);
\draw (4,0) -- (4,1);
\draw (5,0) -- (5,1);
\draw (6,0) -- (6,1);
\draw (7,0) -- (7,1);
\draw (8,0) -- (8,1);
\draw (9,0) -- (9,1);
\draw (10,0) -- (10,1);

\node at (10.5, 0.5) {$\cdots$};

\draw (11,0) -- (14,0);
\draw (11,1) -- (14,1);
\draw (11,0) -- (11,1);
\draw (12,0) -- (12,1);
\draw (13,0) -- (13,1);
\draw (14,0) -- (14,1);

\node[scale=0.8] at (0.5, 1.2) {$1$};
\node[scale=0.8] at (1.5, 1.2) {$2$};
\node[scale=0.8] at (3.5, 1.2) {$2p-3$};
\node[scale=0.8] at (4.5, 1.2) {$2p-2$};
\node[scale=0.8] at (5.5, 1.2) {$2p-1$};
\node[scale=0.8] at (6.5, 1.2) {$2p$};
\node[scale=0.8] at (7.5, 1.2) {$2p+1$};
\node[scale=0.8] at (8.5, 1.2) {$2p+2$};
\node[scale=0.8] at (9.5, 1.2) {$2p+3$};
\node[scale=0.8] at (11.5, 1.2) {$2t-2$};
\node[scale=0.8] at (12.5, 1.2) {$2t-1$};
\node[scale=0.8] at (13.5, 1.2) {$2t$};

\node[scale=0.8] at (0.5, -0.2) {$2$};
\node[scale=0.8] at (1.5, -0.2) {$1$};
\node[scale=0.8] at (3.5, -0.2) {$2p-2$};
\node[scale=0.8] at (4.5, -0.2) {$2p-3$};
\node[scale=0.8] at (5.5, -0.2) {$2p+1$};
\node[scale=0.8] at (6.5, -0.2) {$2p$};
\node[scale=0.8] at (7.5, -0.2) {$2p-1$};
\node[scale=0.8] at (8.5, -0.2) {$2p+3$};
\node[scale=0.8] at (9.5, -0.2) {$2p+2$};
\node[scale=0.8] at (11.5, -0.2) {$2t-1$};
\node[scale=0.8] at (12.5, -0.2) {$2t-2$};
\node[scale=0.8] at (13.5, -0.2) {$2t$};

\draw [fill] (0,0) circle [radius=0.05];
\draw [fill] (1,0) circle [radius=0.05];
\draw [fill] (2,0) circle [radius=0.05];
\draw [fill] (3,0) circle [radius=0.05];
\draw [fill] (4,0) circle [radius=0.05];
\draw [fill] (5,0) circle [radius=0.05];
\draw [fill] (7,0) circle [radius=0.05];
\draw [fill] (14,0) circle [radius=0.05];
\draw [fill] (0,1) circle [radius=0.05];
\draw [fill] (1,1) circle [radius=0.05];
\draw [fill] (2,1) circle [radius=0.05];
\draw [fill] (3,1) circle [radius=0.05];
\draw [fill] (4,1) circle [radius=0.05];
\draw [fill] (5,1) circle [radius=0.05];
\draw [fill] (7,1) circle [radius=0.05];
\draw [fill] (14,1) circle [radius=0.05];

\draw (6,0) circle [radius=0.05];
\draw (12,0) circle [radius=0.05];
\draw (8,0) circle [radius=0.05];
\draw (9,0) circle [radius=0.05];
\draw (10,0) circle [radius=0.05];
\draw (11,0) circle [radius=0.05];
\draw (13,0) circle [radius=0.05];
\draw (6,1) circle [radius=0.05];
\draw (12,1) circle [radius=0.05];
\draw (8,1) circle [radius=0.05];
\draw (9,1) circle [radius=0.05];
\draw (10,1) circle [radius=0.05];
\draw (11,1) circle [radius=0.05];
\draw (13,1) circle [radius=0.05];

\end{tikzpicture}
\end{center}

Here, the gray portions represent the chunks of the cylinder that are glued in the criss-cross pattern, and the white portions are glued in a different pattern. We see that in this configuration, there are two vertices, with cone angles of $4\pi  p$ and $4\pi q$. Therefore, this surface is in $\calH(2p-1, 2q-1)$. 

Now, to get a surface in $\calH(\alpha_1, \alpha_2, \ldots, \alpha_s)$, we can suppose that, up to reordering, the $\alpha_i$'s are numbered so that for each $i$, either $\alpha_i$ is even, or both $\alpha_i$ and exactly one of its neighbors is odd. This is because the sum of the $\alpha_i$ must be even. Then, we may create a surface in this stratum by stringing together the representative permutations for each piece (one even $\alpha_i$ or two odd $\alpha_i$'s next to one another) that we created before, possibly separated by identity permutations on any number squares. For example, $(1)(2,3)(4)(5,7)(6)(8)$ would give us an element in $\calH(2,1,1)$ since $(1)(2,3)$ gives an element in $\calH(2)$, $(5,7)(6)(8)$ is in $\calH(1,1)$, and $(4)$ is a space permutation. We notice that since the representative permutation is $\alpha_i + 1$ squares long for $\alpha_i$ even and $\alpha_i + \alpha_j + 2$ squares long for $\alpha_i, \alpha_j$ both odd, we can create a $1$-cylinder surface in $\calH(\alpha_1, \alpha_2, \ldots, \alpha_s)$ with $n$ squares as long as $\sum_{i=1}^s (\alpha_i + 1) = 2g-2+s \leq n$. 
\end{proof}

Next we show that given a stratum, an STS in the stratum is a holonomy torus if and only if it is the smallest (in terms of number of squares) STS for that stratum.

\begin{proposition}\label{prop:holsquares} $O \in \calH(\alpha_1, \dots, \alpha_s)$ is a holonomy torus if and only if $O$ has $2g-2+s$ squares.
\end{proposition}
\begin{proof} 
$(\Rightarrow)$
If $O$ is a holonomy torus with $n$ squares, by Proposition \ref{prop:althol} every corner of every square is a singularity. So, $2\pi n$ is the total cone angle around the singularities. In any stratum $\calH(\alpha_1, \ldots, \alpha_s)$, the total cone angle around the singularities is $2\pi \sum_{i=1}^s (\alpha_i+1) $. This implies that $O$ must have $(\alpha_1 + \ldots + \alpha_s + s)$ squares. 

$(\Leftarrow)$
If $O$ has  $\sum_{i}^s(\alpha_i+1) = 2g-2+s$ squares, then the total angle at the corners of squares is $2 \pi \sum_{i}^s(\alpha_i+1)$. Note that total angle at the corners of squares = angle coming from the singularities + angle coming from non-singular square corners. However, since $O \in \calH(\alpha_1, \dots, \alpha_s)$, the angle coming from the singularities is $2\pi  \sum_{i}^s(\alpha_i+1)$. Hence, the angle coming from non-singular square corners must be 0 which implies that there are no non-singular square corners so that every corner of every square is a singularity. Then by Proposition \ref{prop:althol} we are done.
\end{proof}

Using this proposition we can obtain an exponential bound in the genus of the surface as an upper bound on the number of holonomy tori in each stratum. 

Next we give an upper bound on the number of squares, as a function of the stratum, below which an STS is forced to be a visibility torus.

\begin{proposition}\label{prop:visguaranteed} For a given stratum $\calH(\alpha_1, \dots, \alpha_s)$, all STSs in the stratum with $\leq 4g - 5 + 2s$ squares are visibility tori. Moreover, this bound is sharp: i.e. there exists a non-visibility STS with $4g-4+2s$ squares.

\end{proposition}
\begin{proof}
We first show that if a square-tiled surface $O \in \calH(\alpha_1, \dots, \alpha_s)$ has fewer than $2s + 2 \sum \alpha_i$ squares then $O$ has a unit horizontal saddle connection. Assume $O$ does not have a unit horizontal saddle connection. Since $O \in \calH(\alpha_1, \dots, \alpha_s)$, $O$ has cone points with angles $2\pi(\alpha_1+1), \dots, 2 \pi(\alpha_s+1)$. Consider an arbitrary cone point of $O$ with angle $2\pi(\alpha_j+1)$. There are $\alpha_j+1$ distinct regular points (with angle $2\pi$) that are square corners, one unit to the right of this cone point. 

So, $O$ has at least $\sum_{i=1}^s(\alpha_i+1)$ points with angle $2 \pi$ that are corners of squares. Hence, the total angle (contributed by the regular corners of squares and the cone points) is at least $ 2 \pi \sum_{i=1}^s (\alpha_i+1) + 2 \pi \sum_{i=1}^s(\alpha_i+1)$. On the other hand, if $O$ has $n$ squares, then this total is $2 \pi n$. Therefore, 
$$ 2 \pi n \geq 2 \pi \sum_{i=1}^s (\alpha_i+1) + 2 \pi \sum_{i=1}^s(\alpha_i+1) \implies n \geq 2s + 2\sum_{i=1}^s \alpha_i$$

Now given $O \in \calH(\alpha_1, \dots, \alpha_s)$ with $\leq 2s + 2\sum_{i=1}^s \alpha_i - 1 = 4g-4+2s-1$ squares, every STS in its $\SL_2\ZZ$ orbit has the same number of squares, and hence  has a unit horizontal saddle connection. By Proposition \ref{prop:altunobs}, $O$ is a visibility torus.

For the other inequality, we will give a construction of a non-visibility and reduced STS in each $\calH(\alpha)$ with $4g-4+2s$ squares. To do so, we start with a one-cylinder surface in $\calH(\alpha)$ with $2g-2+s$ squares, following the construction in the proof of Proposition \ref{prop:1cyl} with no spacer squares. In this surface, every vertex of every square is a cone point, and there is at least one square for which the top and bottom sides are identified. We can then create a new reduced STS in $\calH(\alpha)$ with $4g-4+4s$ squares by splitting this square with top and bottom edges identified into $2g-1+s$ squares. This is demonstrated in the schematic diagram below, where the gray square is the square that we will split, and all cone points are marked with a dot. 

\begin{center}
\begin{tikzpicture}
\draw[fill = lightgray] (0,0) -- (1,0) -- (1,1) -- (0,1) -- (0,0); 
\draw (0,0) -- (4,0) -- (4,1) -- (0,1) -- (0,0); 
\draw (1,0) -- (1,1);
\draw (2,0) -- (2,1);
\draw (3,0) -- (3,1);
\draw[->] (4.5,0.5) -- (5.5, 0.5); 
\draw[fill = lightgray] (6,0) -- (11,0) -- (11,1) -- (6,1) -- (6,0); 
\draw (6,0) -- (14,0) -- (14,1) -- (6,1) -- (6,0); 
\foreach \i in {7,...,13}
	\draw (\i,0) -- (\i,1);
	
\foreach \i in {0,...,4}
	\draw[fill] (\i,0) circle [radius = 0.05];
\foreach \i in {0,...,4}
	\draw[fill] (\i,1) circle [radius = 0.05];

\foreach \i in {11,...,14}
	\draw[fill] (\i,0) circle [radius = 0.05];
\foreach \i in {11,...,14}
	\draw[fill] (\i,1) circle [radius = 0.05];

\draw[fill] (6,0) circle [radius = 0.05];
\draw[fill] (6,1) circle [radius = 0.05];

\node at (8.35,1.8){$2g-1+s$ squares};

\draw[decorate,decoration={brace,amplitude=10pt}]
(6.1,1.2) -- (10.9,1.2);

\node at (12.5,1.8){$2g-3+s$ squares};

\draw[decorate,decoration={brace,amplitude=10pt}]
(11.1,1.2) -- (13.9,1.2);
\end{tikzpicture} 
\end{center} 
The resulting STS has $4g-4+2s$ squares, is still in $\calH(\alpha)$, and is reduced since it has $(1,0)$ and $(0,1)$ as holonomy vector. However, the STS is non-visibility because it does not have $(2g-2+s, 1)$ as a holonomy vectors. 
\end{proof}

We remark here that combined with Proposition \ref{prop:holsquares}, we see that all STSs with number of squares $2g-2+s < n < 4g-4+2s$ have $\Prim$ strictly contained in their holonomy sets. We can also bound the maximum number of squares a visibility torus in a stratum may have:

\begin{proposition}\label{prop:visupperbound} Given a stratum $\calH(\alpha)$, there exists $N(\alpha)$ such that all STSs with $n > N(\alpha)$ number of squares are non-visibility STSs. 
\end{proposition}

\begin{proof}
Fix a stratum $\calH(\alpha)$. It is well known that if $\Prim_{r}$ is the set of primitive vectors in $\ZZ^2$ of length $\leq r$, then $$|\Prim_{r} | = \frac{6}{\pi} r^2 + O(r^{1}) .$$ Let $O$ be a visibility torus in $STS_n \cap \calH(\alpha)$.  Then, if $\Hol(O)_r$ denotes the set of holonomy vectors of $O$ of length $\leq r$, it follows that there exists $c_1 >0$ and $R > 0$ such that, for all $r \geq R$,  $$|\Hol(O)_r| \geq c_1r^2.$$

By Theorem \ref{thm:dozier}, there is a constant $c_2 > 0$ depending only on the stratum such that $|\Hol(X)_r| \leq c_2 r^2$ for all $r$ greater than some constant $R_0(X)$ depending on $X$. In this theorem, all surfaces are assumed to have area $1$. Then, if $O \in \STS_n \cap \calH(\alpha)$, we can scale $O$ by $1/\sqrt{n}$ to get an area one surface. Applying the theorem then gives us that the scaled $O$ has $\leq c_2 (r/\sqrt{n})^2$ saddle connections of length $\leq r/\sqrt{n}$. This implies that for all $r \geq R_0(O)$, we have the following inequality on the original surface: $$|\Hol(O)_r| \leq \frac{c_2 r^2} {n}.$$

Let $N = \frac{c_2}{c_1}$. Then, if $O \in STS_n \cap \calH(\alpha)$ is a visibility torus, we have that $$c_1 r^2 \leq |\Hol(O)_r| \leq \frac{c_2 r^2}{n}$$ which implies  $c_1 r^2 \leq c_2 r^2/n$ for all large enough $r$, giving us that $n \leq c_2/c_1$. This cannot hold for $n > N = c_2/c_1$. Therefore, there are no visibility square-tiled surfaces with  more than $N$ squares in $\calH(\alpha)$. 
\end{proof}

Combining Propositions \ref{prop:1cyl}, \ref{prop:holsquares}, \ref{prop:visguaranteed}, and \ref{prop:visupperbound} we get our theorem:
\faketori

Noting that number of square-tiled surfaces in a given stratum $\calH(\alpha)$ with fewer than or equal to $N(\alpha)$ squares is finite, we conclude that there are finitely many visibility tori in $\calH(\alpha)$.

\begin{remark}
Let $O \in \calH(\alpha)$ with $n$ squares. The best constants $c_1$ and $c_2$ in the proof of Proposition \ref{prop:visupperbound} can be stated as follows:
$$ c_1 = \sup \{ c > 0 : |\Prim_{r}| \geq cr^2  \text{ for all } r\geq 1\}$$
$$ c_2 = \inf \{ c(\alpha) >0 : |\Hol(O)_{r}| \leq \frac{c(\alpha)r^2}{n} \text{for all } O \in \calH(\alpha), \text{ for all } r \geq R(O)\}$$
where $R(O)$ is a constant depending on $O$. 
\end{remark}

Since the number of visibility tori in any given stratum is finite, given a stratum $\calH(\alpha)$, we define 
$$M(\alpha):= \min\{N: \text{ for all } n > N, O \in \STS_n \cap \calH(\alpha) \text{ implies } O \text{ is a non-visibility STS}\}$$
We note here that the constant $M(\alpha)$ is independent of whether we are working with all STSs or just reduced STSs, and that the constant from Theorem \ref{faketori}, $N(\alpha) \geq M(\alpha)$.

The following table records the threshold values of number of squares, stated in Theorem \ref{faketori} for some early strata (genus 2 and 3).
$$
\begin{array}{c|c|c|c|c}\label{boundsonsquares}
\calH(\alpha) &2g+s-2& 4g+2s-5 & \text{size of largest known vis. torus} &  \text{sizes known not to have vis. torus} \\
\hline
\calH(2)& 3 & 5 & 5 & 6-55\\
\calH(1,1) &4 & 7 & 9& 10-52\\
\calH(4) &5& 9 & 10 & 11 - 33 \\
\calH(3,1)&6 & 11 & 15&13, 14, 16 - 19\\
\calH(2,2) &6& 11 & 18& 13, 16, 17, 19 - 23\\
\calH(2,1,1) &7& 13 & 20 &17\\
\calH(1,1,1,1) &8&15 & 25 &17\\
\end{array} 
$$

Using the Sage package {\sf Surface Dynamics} \cite{surfacedyn}, we show the largest visibility tori and smallest non-visibility STSs found in some early strata and illustrated in Figure \ref{fig:largesmallexamples}. 

%Even though the gap between the size of the largest known visibility torus in a stratum and the size ($4g+2s-5$) below which we are guaranteed visibility is significant for some of the strata, we note that even between these bounds, there are some values of $n$ for which we do not see any visibility torus at all, as exhibited in the table.

%\newpage

\begin{figure}

\begin{tikzpicture}[scale=0.95]

\node at (-1,0){$\calH(2)$} ;
\node at (-5, -1){\begin{tikzpicture}[scale=0.58]
\draw (0,0) -- (5,0) -- (5,1) -- (0,1) -- (0,0);
\draw (0,1) -- (0,2) -- (1,2) -- (1,0); 
%\draw (0,0) -- (2,0) -- (2,2) -- (0,2) -- (0,0); 
\draw (2,0) -- (2,1); 
\draw (3,0) -- (3,1); 
\draw (4,0) -- (4,1);

\draw [fill] (0,0) circle [radius=0.15];
\draw [fill] (1,0) circle [radius=0.15];
\draw [fill] (1,1) circle [radius=0.15];
\draw [fill] (0,1) circle [radius=0.15];
\draw [fill] (1,2) circle [radius=0.15];
\draw [fill] (0,2) circle [radius=0.15];
\draw [fill] (5,0) circle [radius=0.15];
\draw [fill] (5,1) circle [radius=0.15];
\end{tikzpicture} };

\node at (4, -1){\begin{tikzpicture}[scale=0.58]
\draw (0,0) -- (4,0) -- (4,1) -- (0,1) -- (0,0);
\draw (0,1) -- (0,2) -- (1,2) -- (1,0); 
\draw (2,0) -- (2,1); 
\draw (3,0) -- (3,1); 
\draw [fill] (0,0) circle [radius=0.15];
\draw [fill] (1,0) circle [radius=0.15];
\draw [fill] (1,1) circle [radius=0.15];
\draw [fill] (0,1) circle [radius=0.15];
\draw [fill] (1,2) circle [radius=0.15];
\draw [fill] (0,2) circle [radius=0.15];
\draw [fill] (4,0) circle [radius=0.15];
\draw [fill] (4,1) circle [radius=0.15];
\end{tikzpicture}};
\node at (-5, -2){$n =6$, primitive, Veech group index = 36};
\node at (4, -2){$n=5$, primitive, Veech group index = 18};

\node at (-1, -3){$\calH(1,1)$};
\node at (-5, -4){\begin{tikzpicture}[scale = 0.55]
\draw (0,0) -- (5,0) -- (5,1) -- (0,1) -- (0,0);
\draw (0,1) -- (-2,1) -- (-2, 2) -- (1,2) -- (1,0);
%\draw (0,1) -- (0,2) -- (1,2) -- (1,0); 
%\draw (0,0) -- (2,0) -- (2,2) -- (0,2) -- (0,0); 
\draw (2,0) -- (2,1); 
\draw (3,0) -- (3,1); 
\draw (4,0) -- (4,1);
\draw (0,1) -- (0,2);
\draw (-1,1) -- (-1, 2);
\draw [fill] (0,0) circle [radius=0.15];
\draw (1,0) circle [radius=0.15];
\draw (1,1) circle [radius=0.15];
\draw [fill] (0,1) circle [radius=0.15];
\draw  (1,2) circle [radius=0.15];
\draw [fill] (0,2) circle [radius=0.15];
\draw [fill] (5,0) circle [radius=0.15];
\draw [fill] (5,1) circle [radius=0.15];
\draw (-2, 1) circle [radius=0.15];
\draw (-2,2) circle [radius=0.15];
\end{tikzpicture}};

\node at (4, -4){\begin{tikzpicture}[scale=0.58]
\draw (0,0) -- (6,0) -- (6,1) -- (0,1) -- (0,0);
\draw (0,1) -- (-2,1) -- (-2, 2) -- (1,2) -- (1,0);
%\draw (0,1) -- (0,2) -- (1,2) -- (1,0); 
%\draw (0,0) -- (2,0) -- (2,2) -- (0,2) -- (0,0); 
\draw (2,0) -- (2,1); 
\draw (3,0) -- (3,1); 
\draw (4,0) -- (4,1);
\draw (5,0) -- (5,1);
\draw (0,1) -- (0,2);
\draw (-1,1) -- (-1, 2);

\draw [fill] (0,0) circle [radius=0.15];
\draw (1,0) circle [radius=0.15];
\draw (1,1) circle [radius=0.15];
\draw [fill] (0,1) circle [radius=0.15];
\draw  (1,2) circle [radius=0.15];
\draw [fill] (0,2) circle [radius=0.15];
\draw [fill] (6,0) circle [radius=0.15];
\draw [fill] (6,1) circle [radius=0.15];
\draw (-2, 1) circle [radius=0.15];
\draw (-2,2) circle [radius=0.15];

\end{tikzpicture}};

\node at (-5, -5){$n = 8, \text{ primitive, Veech group index}= 144$};
\node at (3.5, -5){ $n = 9, \text{ not primitive, Veech group index}= 48$};

\node at (-1, -6){$\calH(4)$};

\node at (-5, -7){\begin{tikzpicture}[scale=0.58 ]
\draw (0,0) -- (5,0) -- (5,1) -- (0,1) -- (0,0);
\draw (0,1) -- (-4,1) -- (-4, 2) -- (1,2) -- (1,0);
\draw (2,0) -- (2,1); 
\draw (3,0) -- (3,1); 
\draw (4,0) -- (4,1);
\draw (0,1) -- (0,2);
\draw (-1,1) -- (-1, 2);
\draw (-2,1) -- (-2, 2);
\draw (-3,1) -- (-3, 2);
%\draw (-4,1) -- (-4, 2);

\draw [fill] (0,0) circle [radius=0.15];
\draw [fill](1,0) circle [radius=0.15];
\draw [fill] (1,1) circle [radius=0.15];
\draw [fill] (0,1) circle [radius=0.15];
\draw [fill] (1,2) circle [radius=0.15];
\draw [fill] (0,2) circle [radius=0.15];

\draw [fill] (0,2) circle [radius=0.15];
%\draw [fill] (0,2) circle [radius=0.05];
\draw [fill] (5,0) circle [radius=0.15];
\draw [fill] (5,1) circle [radius=0.15];
\draw [fill]( -2, 1) circle [radius=0.15];
\draw [fill] (-2,2) circle [radius=0.15];

\draw [fill]( -4, 1) circle [radius=0.15];
\draw [fill] (-4,2) circle [radius=0.15];

\node at (-3.5, 0.7) {$1$};
\node at (-2.5, 0.7) {$2$};
\node at (-1.5, 0.7) {$3$};
\node at (-0.5, 0.7) {$4$};

\node at (-3.5, 1.4+0.9) {$3$};
\node at (-2.5, 1.4+0.9) {$4$};
\node at (-1.5, 1.4+0.9) {$1$};
\node at (-0.5, 1.4+0.9) {$2$};

\end{tikzpicture}};

\node at (4, -7){ \begin{tikzpicture}[scale=0.58]
\draw (0,0) -- (7,0) -- (7,1) -- (0,1) -- (0,0);
\draw (0,1) -- (-2,1) -- (-2, 2) -- (1,2) -- (1,0);
\draw (2,0) -- (2,1); 
\draw (3,0) -- (3,1); 
\draw (4,0) -- (4,1);
\draw (5,0) -- (5,1);
\draw (6,0) -- (6,1);

\draw (0,1) -- (0,2);
\draw (-1,1) -- (-1, 2);
%\draw (-2,1) -- (-2, 2);
%\draw (-3,1) -- (-3, 2);
%%\draw (-4,1) -- (-4, 2);

\draw [fill] (0,0) circle [radius=0.15];
\draw [fill](1,0) circle [radius=0.15];
\draw [fill] (1,1) circle [radius=0.15];
\draw [fill] (0,1) circle [radius=0.15];
\draw [fill] (1,2) circle [radius=0.15];
\draw [fill] (0,2) circle [radius=0.15];

\draw [fill] (0,2) circle [radius=0.15];
%\draw [fill] (0,2) circle [radius=0.05];
\draw [fill] (7,0) circle [radius=0.15];
\draw [fill] (7,1) circle [radius=0.15];
\draw [fill]( -2, 1) circle [radius=0.15];
\draw [fill] (-2,2) circle [radius=0.15];

\draw [fill]( -1, 1) circle [radius=0.15];
\draw [fill] (-1,2) circle [radius=0.15];

\node at (-1.5, 0.7) {$1$};
\node at (-0.5, 0.7) {$2$};

\node at (-1.5, 1.4+0.9) {$2$};
\node at (-0.5, 1.4+0.9) {$1$};

\end{tikzpicture}};

\node at (-5, -8.2){$n = 10, \text{ primitive, Veech group index}= 450$};

\node at (4, -8.2){$n = 10, \text{ primitive, Veech group index}= 1065$};

\node at (-1,-9){$\calH(3,1)$};
 
\node at (-5.5, -10){\begin{tikzpicture}[scale=0.58]
\draw (0,0) -- (12,0) -- (12,1) -- (0,1) -- (0,0);
\foreach \x in {1,2,3,4,5,6,7,8,9,10,11}
	\draw (\x,0) -- (\x,1); 

\foreach \x in {0, 1, 2}{
	\draw [fill] (\x,0) circle [radius=0.15];
	\draw [fill] (\x,1) circle [radius=0.15];
	}
\foreach \x in {3, 5}{
	\draw  (\x,0) circle [radius=0.15];
	\draw  (\x,1) circle [radius=0.15];
	}
	
\foreach \x in {4, 12}{
	\draw [fill] (\x,0) circle [radius=0.15];
	\draw [fill] (\x,1) circle [radius=0.15];
	}

\node at (0.5, -0.3 ) {$1$};
\node at (1.5, -0.3) {$2$};
\node at (1.5, 1.5-0.2 ) {$1$};
\node at (0.5, 1.5-0.2) {$2$};

\node at (2.5, -0.3 ) {$3$};
\node at (4.5, -0.3) {$5$};
\node at (4.5, 1.5-0.2 ) {$3$};
\node at (2.5, 1.5-0.2) {$5$};

\end{tikzpicture}};

\node at (3.5, -10){\begin{tikzpicture}[scale=0.58]
\draw (0,0) -- (6,0) -- (6,1) -- (0,1) -- (0,0);
\draw (0,1) -- (-7, 1) -- (-7, 2) -- (2, 2) -- (2, 1);
\foreach \x in {1,2,3,4,5}
	\draw (\x,0) -- (\x,1); 
	
\foreach \x in {1,0,-1, -2, -3, -4, -5, -6}
	\draw (\x,1) -- (\x,2); 

\foreach \x in {2, 5}{
	\draw [fill] (\x,0) circle [radius=0.15];
	\draw [fill] (\x,1) circle [radius=0.15];
	}
\draw [fill] (2,2) circle [radius=0.15];
\draw [fill] (-7,1) circle [radius=0.15];
\draw [fill] (-7,2) circle [radius=0.15];

\draw [fill] (-1,1) circle [radius=0.15];
\draw [fill] (-1,2) circle [radius=0.15];

\draw (0,1) circle [radius=0.15];
\draw (3,1) circle [radius=0.15];
\draw (3,0) circle [radius=0.15];

\draw (6,1) circle [radius=0.15];

\node at (-0.5, 0.9-0.2 ) {$1$};
\node at (2.5, 1.3) {$1$};

\node at (2.5, -0.3 ) {$2$};
\node at (5.5, 1.3 ) {$2$};

\node at (-0.5, 2.3 ) {$3$};
\node at (5.5, -0.3 ) {$3$};

\end{tikzpicture}};

\node at (-5, -11.2){$n=12, \text{ primitive, Veech group index} = 86784$};
\node at (3.5, -11.2){$n =15, \text{ not primitive, Veech group index} = 3072$};

\node at (-1, -12){$\calH(2,2)$};

\node at (-6, -13.5){\begin{tikzpicture}[scale=0.58]
\draw (0,0) -- (4,0) -- (4,1) -- (0,1) -- (0,0);
\draw (1,1) -- (1, 2) -- (5, 2) -- (5, 1) -- (4, 1);
\draw (1,2) -- (0, 2) -- (0, 3) -- (4, 3) -- (4, 2);

\foreach \x in {2,3}
	\draw (\x,0) -- (\x,3); 
\foreach \x in {0,2}
	\draw (1,\x) -- (1,\x+1); 
	
\draw (4,1) -- (4, 2);

\foreach \x in {0,1,2,3}{
	\draw [fill](1,\x) circle [radius=0.15];
	\draw (4,\x) circle [radius=0.15];
	\draw (0,\x) circle [radius=0.15];
	}
\foreach \x in {1,2}
	\draw [fill](5,\x) circle [radius=0.15];
\end{tikzpicture}};

\node at (2.75, -13.5){\begin{tikzpicture}[scale=0.58]
\draw (0,0) -- (15,0) -- (15,1) -- (0,1) -- (0,0);
\draw (0,1) -- (-1, 1) -- (-1, 2) -- (2, 2) -- (2, 1);
\foreach \x in {1,2,3,4,5, 6, 7, 8, 9, 10, 11, 12, 13, 14}
	\draw (\x,0) -- (\x,1); 
\foreach \x in {0, 1}
	\draw (\x,1) -- (\x,2); 
	
\foreach \x in {0, 6, 15}{
	\draw (\x,0) circle [radius=0.15];
	\draw (\x,1) circle [radius=0.15];
	}

\foreach \x in {2, 5}{
	\draw [fill](\x,0) circle [radius=0.15];
	\draw [fill] (\x,1) circle [radius=0.15];
	}
\draw [fill](-1,1) circle [radius=0.15];
\draw [fill] (-1,2) circle [radius=0.15];	
\draw [fill] (2,2) circle [radius=0.15];	

\draw (0,2) circle [radius=0.15];

\node at (-0.5, 1-0.3 ) {$1$};
\node at (5.5, 1.3 ) {$1$};

\node at (-0.5, 2.3 ) {$2$};
\node at (5.5, -0.3 ) {$2$};

\end{tikzpicture}};

\node at (-5.5, -14.8){$n = 12, \text{ not primitive, Veech group index} = 6$};
\node at (3.5, -14.8){$n = 18, \text{ not primitive, Veech group index} = 4320$};

\node at (-1, -15.6){$\calH(2,1,1)$};

\node at (-5, -16.8){\begin{tikzpicture}[scale=0.58] 
\draw (0,0) -- (14,0) -- (14,1) -- (0,1) -- (0,0);
\foreach \x in {1,2,3,4,5, 6, 7, 8, 9, 10, 11, 12, 13}
	\draw (\x,0) -- (\x,1); 
	
\foreach \x in {0, 12,13,14}{
	\draw [fill](\x,0) circle [radius=0.15];
	\draw [fill] (\x,1) circle [radius=0.15];
	}
	
\foreach \x in {1, 3}{
	\draw (\x,0) circle [radius=0.15];
	\draw  (\x,1) circle [radius=0.15];
	}

\foreach \x in {2, 4}{
	\draw (\x-0.2,-0.2) rectangle (\x+0.2, 0.2);
	\draw (\x-0.2,1-0.2) rectangle (\x+0.2, 1+0.2);

	}
	
\node at (1.5, -0.3 ) {$1$};
\node at (3.5, 1.3 ) {$1$};

\node at (3.5, -0.3 ) {$2$};
\node at (1.5, 1.3 ) {$2$};

\node at (12.5, -0.3 ) {$3$};
\node at (13.5, 1.3 ) {$3$};

\node at (13.5, -0.3 ) {$4$};
\node at (12.5, 1.3 ) {$4$};

\end{tikzpicture}};
\node at (4, -16.8){\begin{tikzpicture}[scale=0.58] 
\draw (0,0) -- (0,2) -- (8,2) -- (8,0) -- (1,0) -- (1,-1) -- (-3,-1) -- (-3,0) -- (0,0);
\draw (0,1) -- (8,1);
\draw (0,0) -- (1,0);
\foreach \x in {1,2,3,4,5, 6, 7 }
	\draw (\x,0) -- (\x,2); 
\foreach \x in {0,-1,-2}
	\draw (\x,0) -- (\x,-1);

\foreach \x in {0, 4,8}{
	\foreach \y in {0,2}{
	\draw [fill](\x,\y) circle [radius=0.15];
	}
	}
	
\foreach \y in {0, -1,2}{
	\draw (1,\y) circle [radius=0.15];
	
	}
\foreach \y in {0, -1}{
	\draw (-3,\y) circle [radius=0.15];
	
	}

\foreach \y in {0,2}{
	\draw (3-0.2,\y-0.2) rectangle (3+0.2, \y+0.2);
	}

\foreach \y in {0,-1}{
	\draw (-1-0.2,\y-0.2) rectangle (-1+0.2, \y+0.2);
	}
	
\draw [fill] (0,-1) circle [radius = 0.15];
\node at (-0.5, -1.3 ) {$1$};
\node at (3.5, 2.3 ) {$1$};

\node at (3.5, -0.3 ) {$2$};
\node at (-0.5, 0.3 ) {$2$};

\end{tikzpicture}};

\node at (-5, -18.3){$n =14, \text{ primitive, Veech group index}= 778968$};
\node at (3.5, -18.3){$n =20, \text{ not primitive, Veech group index}= 36$};

\node at (-1, -19.3){$\calH(1,1,1,1)$};

\node at (-5, -20.7){\begin{tikzpicture}[scale=0.58]
\draw (0,0) -- (14,0) -- (14,1) -- (0,1) -- (0,0);
\draw (0,1) -- (-1, 1) -- (-1, 2) -- (1, 2) -- (1, 1);
\foreach \x in {1,2,3,4,5, 6, 7, 8, 9, 10, 11, 12, 13}
	\draw (\x,0) -- (\x,1); 
\draw (0,1) -- (0,2);

\foreach \x in {0, 14}{
	\draw [fill](\x,0) circle [radius=0.15];
	\draw [fill] (\x,1) circle [radius=0.15];
	}
\draw [fill] (1,2) circle [radius = 0.15];
\draw [fill] (-1,2) circle [radius = 0.15];

\foreach \x in {0, 1}{
	\draw (1,\x) circle [radius = 0.15];
	\draw (0, \x) circle [radius = 0.15];
	}
\draw (-1,1) circle [radius = 0.15];
\draw (0,2) circle [radius = 0.15];

\foreach \x in {2, 4}{
	\draw (\x-0.2,-0.2) rectangle (\x+0.2, 0.2);
	\draw (\x-0.2,1-0.2) rectangle (\x+0.2, 1+0.2);
}

\foreach \x in {3, 5}{
	\draw [fill](\x-0.2,-0.2) rectangle (\x+0.2, 0.2);
	\draw [fill](\x-0.2,1-0.2) rectangle (\x+0.2, 1+0.2);
}

\node at (-0.5, 1-0.3 ) {$1$};
\node at (0.5, 2.3 ) {$1$};

\node at (0.5, -0.3 ) {$2$};
\node at (-0.5, 2.3 ) {$2$};

\node at (2.5, -0.3 ) {$3$};
\node at (4.5, 1.3 ) {$3$};

\node at (4.5, -0.3 ) {$4$};
\node at (2.5, 1.3 ) {$4$};

%
%	
%\foreach \x in {2, 5}{
%	\draw [fill](\x,0) circle [radius=0.15];
%	\draw [fill] (\x,1) circle [radius=0.15];
%	}
%\draw [fill](-1,1) circle [radius=0.15];
%\draw [fill] (-1,2) circle [radius=0.15];	
%\draw [fill] (2,2) circle [radius=0.15];	
%
%\draw (0,2) circle [radius=0.15];
%

\end{tikzpicture}};

\node at (4, -20.7){\begin{tikzpicture}[scale=0.58]
\draw (0,0) -- (10,0) -- (10,1) -- (9,1) -- (9,4) -- (6,4)--(6,5) -- (4,5) -- (4,2) -- (6,2) -- (6,1) -- (0,1) -- (0,0);

\foreach \x in {1,2,3,4,5,6,9}{
	\draw (\x,0) -- (\x, 1); 
	}

\foreach \x in {7,8}
	\draw (\x,0) -- (\x,4); 
\foreach \y in {1,2,3}
	\draw (6,\y) -- (9,\y); 
	
\foreach \x in {5,6} \draw (\x, 2) -- (\x, 4);
\foreach \y in {3,4} \draw (4, \y) -- (6, \y);

\draw (5, 4) -- (5, 5);

\foreach \y in {2,5}{
	\draw [fill](6,\y) circle [radius=0.15];
	
	}
\foreach \y in {0,1}
	\draw [fill](1,\y) circle [radius=0.15];

\foreach \y in {0,1}
	\draw [fill](1,\y) circle [radius=0.15];

\foreach \y in {0,1}{
	\foreach \x in {0, 10}{
	\draw (\x,\y) circle [radius=0.15];
	}}
	
\foreach \y in {2,5}
	\draw (5,\y) circle [radius=0.15];
	
\foreach \y in {0,1,4}{
	\draw (6-0.2,\y-0.2) rectangle (6+0.2, \y+0.2);
	}

\foreach \y in {0,1,4}{
	\draw [fill](9-0.2,\y-0.2) rectangle (9+0.2, \y+0.2);
	}
	
\foreach \y in {2,5}{
	\draw [fill](4-0.2,\y-0.2) rectangle (4+0.2, \y+0.2);
	}

\node at (0.5, 1.3){$1$};
\node at (5.5, 2-0.3){$1$};

\node at (0.5, -0.3){$2$};
\node at (5.5, 5.3){$2$}; 

\node at (6.3, 1.5){$3$};
\node at (6.3, 4.5){$3$};

\foreach \y in {4, 5, 6}{
	\node at (4-0.3, \y - 1.5){\y};
	\node at (9+0.3, \y - 2.5){\y};
}

\end{tikzpicture}};

\node at (-5.5, -22.7){$n= 16, \text{ primitive, Veech group index}= 1387296$};

\node at (3.5, -22.7){$n=25, \text{ not primitive, Veech group index}=49536$};

\foreach \y in {-23,-18.6,-15.1, -11.5,-8.5,  -5.3, -2.3}{
	\draw (-9.5,\y) -- (7.5, \y);
	}
\end{tikzpicture}

\caption{Large and small examples illustrating visibility: For each stratum, the surface on the left is the smallest non-visibility STS, and the surface on the right is the largest visibility torus we have found}
\label{fig:largesmallexamples}

\end{figure}
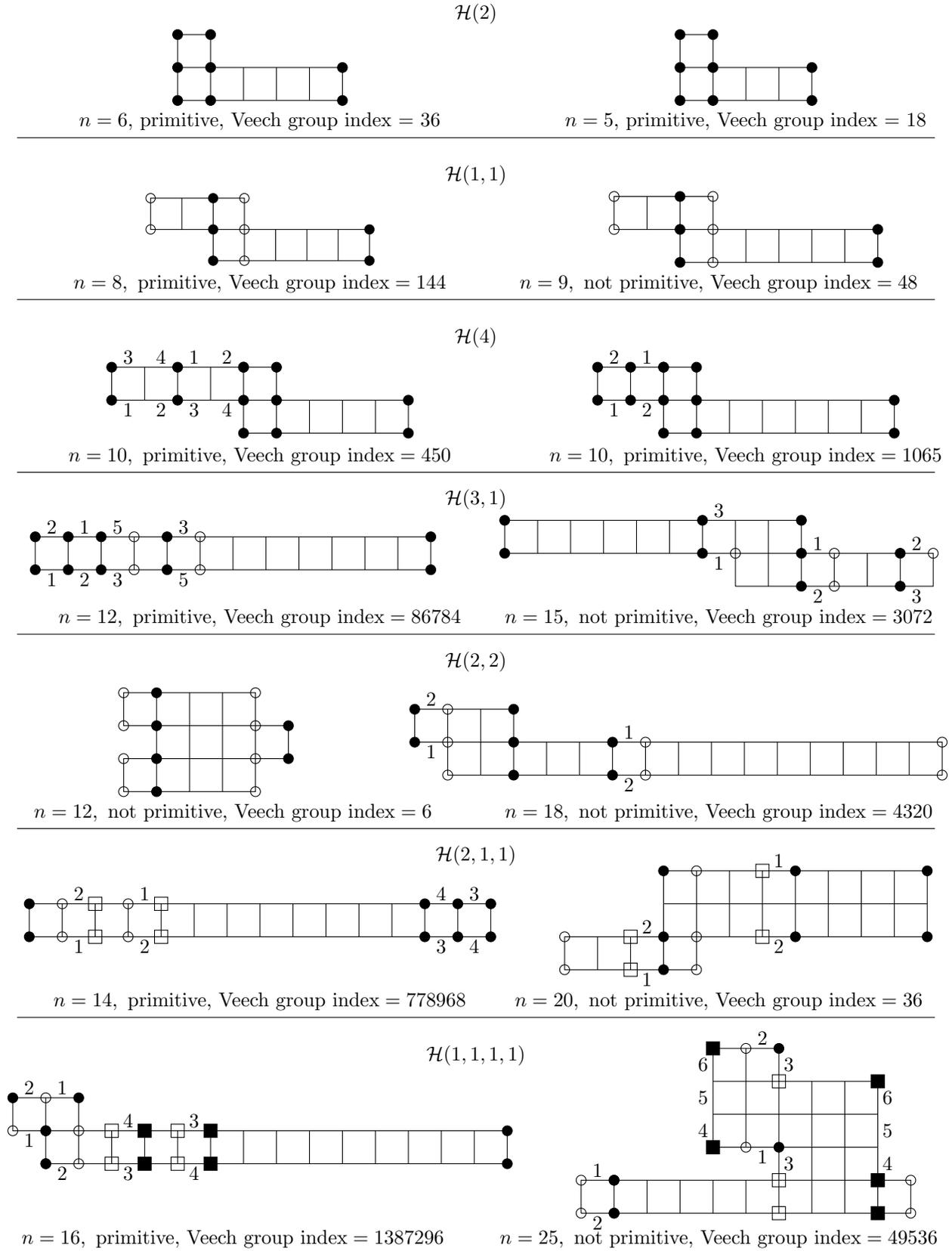

%%%%

\subsection{Counting problems related to symmetry}

Although symmetry tori are not a subset of holonomy tori (or vice versa), we still expect that the number of symmetry tori in each stratum is also finite. As a step toward a proof of this conjecture, we prove that there are no symmetry tori in genus two by analyzing each stratum of genus two separately. 

The fact that $\calH(2)$ does not contain a symmetry torus, follows from the works of Hubert-Leli\` evre \cite{hubertlelievre}, McMullen \cite{mcmullen} and  Leli\` evre-Royer \cite{lelievreroyer}. The first two works show that the reduced STSs in $\calH(2)$ partition into one or two orbits (depending on the parity of the number of squares) under the action of $SL_2\ZZ$. Leli\` evre-Royer then give a formula for the number of STSs in each orbit. For odd $n \geq 3$, there are two orbits $A_n$ and $B_n$ and they show the respective counts to be:
$$ |A_n| = \frac{3}{16}(n-1)n^2 \prod_{p|n} \left(1-\frac{1}{p^2}\right) \hspace{1cm} \text{and} \hspace{1cm} |B_n| = \frac{3}{16}(n-3)n^2 \prod_{p|n} \left(1-\frac{1}{p^2}\right)$$
For even $n \geq 4$, there is only one orbit with $\frac{3}{8} (n-2)n^2 \prod_{p|n} (1-p^{-2})$ STSs. If an STS were a symmetry torus in $\calH(2)$, it would be the only one in its orbit.  Since neither of these orbit counts are 1 for any $n \geq 3$ (which is the minimum number of squares required to be in $\calH(2)$), it follows that there are no symmetry tori in $\calH(2)$. 

The orbit classification of reduced surfaces in $\calH(1,1)$ (and hence the individual orbit counts) is not known. However, one can still show that there are no symmetry tori in $\calH(1,1)$ by analyzing the separatrix diagrams:

\begin{lemma} There are no symmetry tori in $\calH(1,1)$.
\end{lemma}

\begin{proof}
A surface in $\calH(1,1)$ has one of 4 separatrix diagrams, as shown in Figure \ref{fig:H11surfaces}.  We will handle these 4 cases separately. Let $T = \begin{bmatrix}1 & 1 \\ 0 & 1\end{bmatrix}$ and $S = \begin{bmatrix}0 & -1 \\ 1 & 0\end{bmatrix}$. We will show that given a reduced $O \in \calH(1,1)$, both $T$ and $S$ cannot simultaneoulsy fix $O$. We will do the analysis case wise, by separatrix diagram.

Let $O \in \calH(1,1)$ with separatrix diagram A (i.e. $O$ has one horizontal cylinder), parametrized by $(p, j, k, l, m, \alpha)$ as in Figure \ref{fig:H11surfaces}. For $O$ to be reduced, the height of the cylinder, $p$,  needs to be 1. Without  loss of generality, assume the shear parameter $\alpha$ is 0, since, the the surface defined by $(1, j, k, l, m, 0)$, is in the $SL(2, \ZZ)$. Now that $\alpha = 0$, $O$ has a vertical closed saddle connection, and $T \sdot O$ has parameters $(1, j, k, l, m, 1)$ and does not have such a saddle connection unless $k -m = 1$. But in that case, $T \sdot O$ has closed saddle connections with holonomy $(1,1)$ whereas $O$ does not have such closed saddle connections. So, $T \sdot O \neq O$.

Suppose now $O$ has separatrix diagram B and is parametrized by $(p, q, k, l,m, \alpha, \beta)$. Since this set of $p, q, k, l, m, \alpha, \beta \in \NN$ such that $p (k+l+m) + q m = n$, $0 \leq \alpha < n$ and $0 \leq \beta < n$ uniquely parametrizes surfaces in $\STS_n \cap \calH(1,1)$ with separatrix diagram $B$, if $T\sdot O = 0$, then length of the horizontal cylinders must divide the height of the cylinders $(k+l+m) | p$ and $m | q$. Then, assume $S \in SL(O)$ which implies that $O$ has vertical curves which are core curves of cylinders of length $m$ and $k+l+m$. However, any vertical cylinder in $O$ must have length $xp + y(p+q)$ for some $x, y \geq 0$. Hence, $k+l+m = (x+y)p + y q$ for some $x, y \geq 0$ integers. This is a contradiction as $k+l+m | p$.

Now suppose that $O$ has separatrix diagram $C$, and is parametrized by $(p, q, k, l, m, \alpha$, $ \beta)$. Once again, if $T \in SL(O)$ then $(k+l) | p$ and $(l+m) | q$. Moreover, $O$ has a horizontal saddle connection of length $l$ which is shared by the two horizontal cylinders. Then assume $S \in SL(O)$,which implies $O$ must also have a vertical saddle connection of length $l$. However, any vertical saddle connection of $O$ must have length at least $\min \{ p, q\}$, a contradiction since $(l+k) |p$ and $(l+m) | q$.

Finally, suppose that $O$ has separatrix diagram $D$, and is parametrized by $(p, q, r, k, l $, $\alpha, \beta, \gamma)$. If $T \in SL(O)$,then $k |p$, $(l+k)|q$ and $l | r$. Further, assume $S \in SL(O)$, so that $O$ must have a vertical cylinders of length $k$. But any vertical curve that is the core curve of a cylinder has length $ k= x(p+q) + y (q+r)$ for $x, y \geq 0$ integers. But as $k |p$, we have $p = ik$ for some $i \geq 0$ integer, so that $p = ix(p+q) + iy (q+r)$. This implies $x = 0$. Hence, $(q+r) | p$. By a symmetric argument, $(p+q) | r$. This is only possible if $p = r$. However, as $O$ is reduced, $\gcd(p+q, q+r )= 1$, which implies that $\gcd(r+q, q+r)= 1$, and therefore $q+r = 1$, which is a contradiction.
\end{proof}

Hence, as there are no symmetry tori in the two strata that make up genus two, we conclude:

\symmetrytori

%%%%
\section{Counting Surfaces with Unit Saddles}
\label{sec:unitloops}

In this section, we investigate a combinatorial question about the proportion of square-tiled surfaces in $\mathcal{H}(2)$ that have a short saddle connection. We hope that the methods used in this investigation can be adapted to work on other counting problems. 

We define a {\bf unit saddle} as a saddle connection of length one, and we denote the collection of translation surfaces with unit saddles as $\Unit$. In this section, we will show that a random square-tiled surface in $\calH(2)$ has no unit saddle, asymptotically almost surely. In light of Proposition \ref{prop:altunobs}, we note that having a unit saddle is a weaker property than being visibility, since now we just require the surface in question, and not necessarily its entire $SL_2(\ZZ)$ orbit, to have a unit saddle.

Let us first review some asymptotic notations. We say that $f(n)$ is $O(g(n))$ if there exists a constant $C$ and $n_0 \in \NN$ for which $$f(n) \leq C \sdot g(n)$$ for all $n \geq n_0$. We say that $f(n)$ is $\Omega(g(n))$ if there exists a constant $c$ and $n_0 \in \NN$ for which $$ c \sdot g(n) \leq f(n)$$ for all $n \geq n_0$. A function is $\Theta(g(n))$ if it is both $O(g(n))$ and $\Omega(g(n))$. Our main result in this section is the following theorem. 

\total

In this section, we will first present a heuristic argument for why one should expect the asymptotics given in this theorem, as well as some experimental support. We will follow this with a combinatorial proof of the theorem. 

\subsection{Heuristics and experimental support} 

We will first quickly review some background on volumes of strata of translation surfaces. For a more extensive introduction, see \cite{zorich}. Square-tiled surfaces play the role of integer points in a strata $\calH(\alpha)$, and as such, the counts of square-tiled surfaces are intimately related to the volumes and dimensions of these strata. We recall that a translation surface in a given stratum $\calH(\alpha)$ is represented by a pair $(X,\omega)$ where $\omega$ has the degrees of its zeros prescribed by $\alpha$. Integrating $\omega$ along paths connecting the zeros of $\omega$ gives its \textbf{relative periods}, and locally gives a map between $\calH(\alpha)$ and an open subset of the vector space $H^1(\Sigma, \{p_1,\ldots,p_n\}; \CC)$, where $\Sigma$ is the underlying topological surfaces and the $p_i$'s are the locations of the zeros of the one-form. 

There is a natural integer lattice $H^1(\Sigma, \{p_1,\ldots,p_n\}; \ZZ \oplus i \ZZ) \subset H^1(\Sigma, \{p_1,\ldots,p_n\}; \CC).$ The translation surfaces whose period maps land on these integer points are square-tiled surfaces, the integer points of the strata. Normalizing so that the volume of a unit cube is one, there is a natural volume form $d\mu$ on $\calH(\alpha)$ inherited from period coordinates. If we let $S(X,\omega)$ be the volume of the translation surface $(X,\omega)$, the volume form $d\mu$ then induces a volume form $d\mu_1 = d \mu / d S$ on $\calH_1(\alpha)$, the space of volume one translation surfaces in the stratum. It is a theorem of Masur (\cite{masur}) and Veech (\cite{veech2}) that the volumes of $\calH_1(\alpha)$ with respect to $d\mu_1$ are finite. 

Having set up this background, we can how give a heuristic argument for the asymptotics given in Theorem \ref{total}. First, since $\calH(2)$ is a $4$-dimensional space, if we let $v(k)$ denote the number of square-tiled surfaces in $\calH(2)$ with $\leq k$ squares, we should find that $v(k) \sim ck^4$ for some constant $c$. Since $v(k)$ is the total number of $n$-square STSs for all $n \leq k$, we expect the count of $|\STS_n \cap \calH(2)|$ to be on the order of $n^3$, since the derivative of $v(k)$ at $n$ is on the order of $n^3$. 

Furthermore, rescaling a square-tiled surface with $n$ squares to be area one results in a surface tiled with squares of side length $1/\sqrt{n}$  in $\calH_1(2)$. Thus, surfaces in $\STS_n \cap \calH(2)$ form finer and finer mesh grids of $\calH_1(2)$ as $n \rightarrow \infty$. Those square-tiled surfaces that have a unit saddle reside in the cusps of $\calH_1(2)$ of surfaces with shortest saddle connection length $\leq 1/\sqrt{n}$. Since the areas of these cusps go to zero as $n$ goes to infinity, and we expect square-tiled surfaces to equidistribute in $\calH_1(2)$, it then follows that the proportion of square-tiled surfaces in $\calH(2)$ with a unit saddle should be asymptotically almost surely $0$. 

Our result is also supported by experimental evidence. Using the Origami database in the {\sf Surface Dynamics} Sage package (\cite{surfacedyn}), we computed the proportion of reduced square-tiled surfaces with a unit saddle in $\calH(2)$ for up to $55$ squares. Figure~\ref{reduced-plot} shows a plot of those proportions and their reciprocals, suggesting a stronger result than is entailed by
Theorem~\ref{total}:  the proportion seems to be reciprocal-linear, not just bounded above and below by reciprocal linear functions as the $\Theta(1/n)$ asymptotic would suggest.

\begin{figure}[ht]
\centering
\includegraphics[width=.45\textwidth]{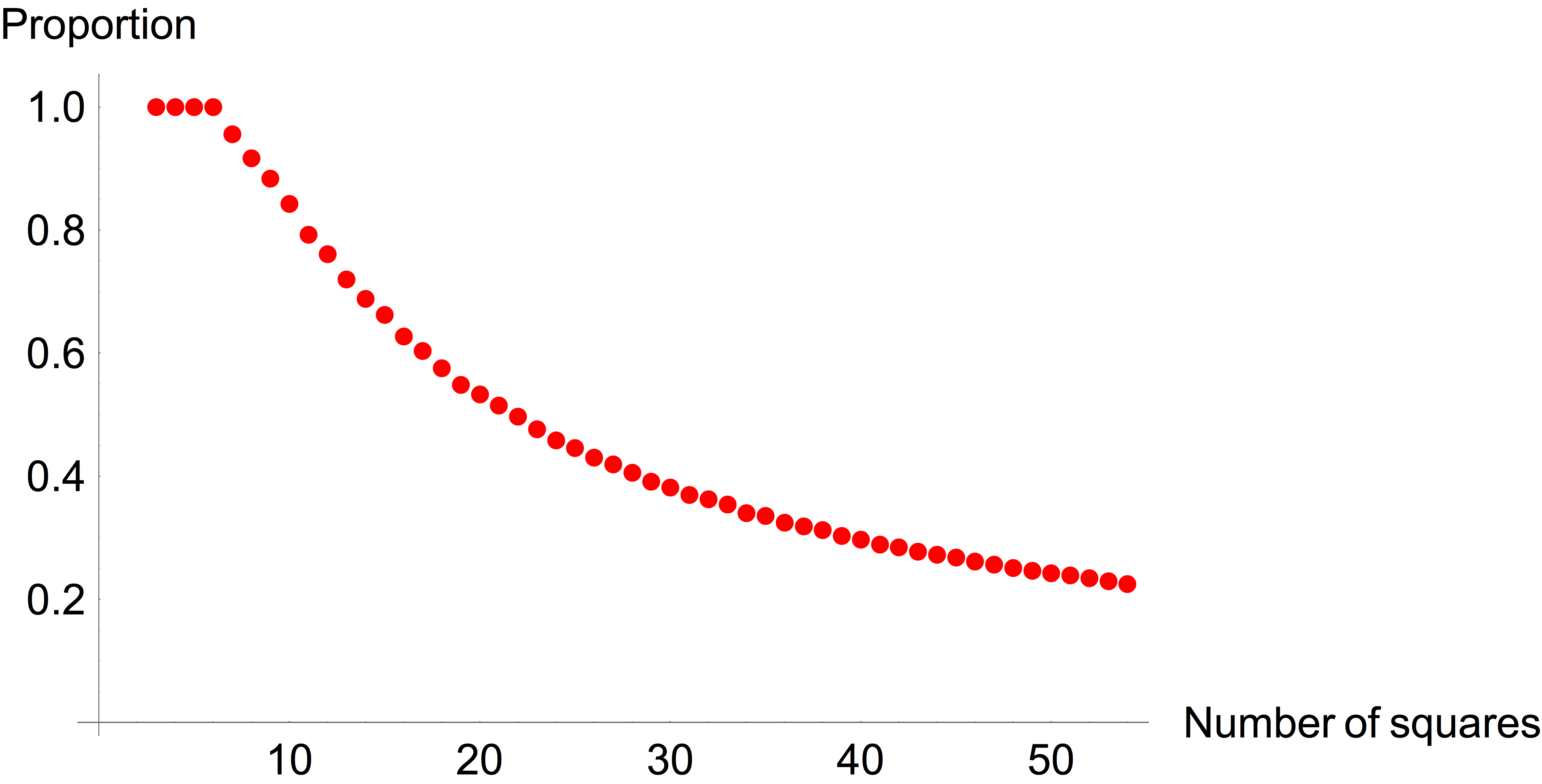} \quad
\includegraphics[width=.45\textwidth]{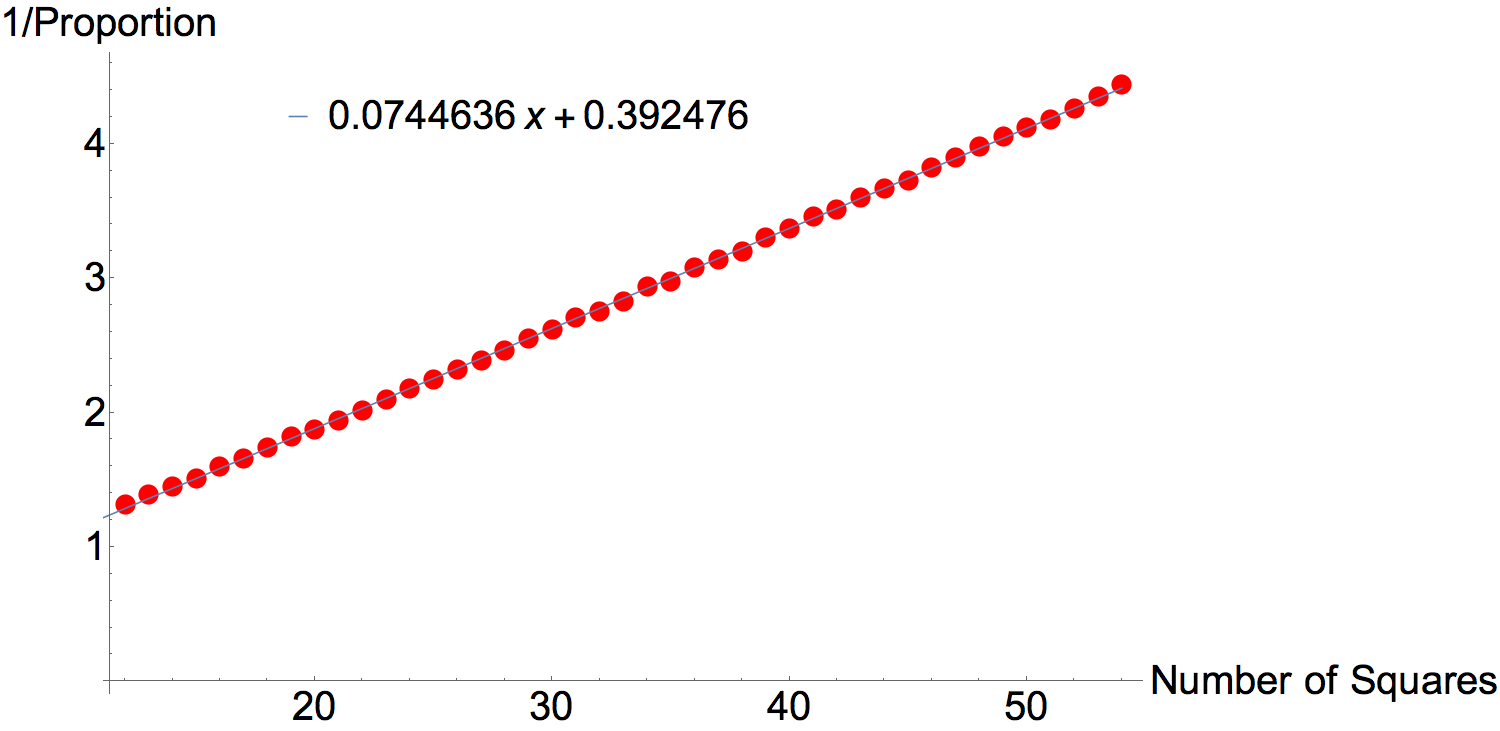}
\caption{Left: The proportion of STSs with a unit saddle in $\calH(2)$ decreases as $n$ increases. Right: The reciprocal of the proportion of STSs with a unit saddle seems to grow linearly with $n$. The equation of the line of best fit is shown.  }
\label{reduced-plot}
\end{figure}

%The fitted curve shown is fitted using Mathematica's NonLinearModelFit function. However, to avoid the fit being skewed by the proportions when the number of squares is small, we have to choose an appropriate cut off value, $N$, for the number of squares, and fit a curve to the proportion values corresponding to the number of squares $N, \dots, 54$.

%To make an appropriate choice for $N$, we examine the fits and the mean squared errors using cut off values $N = 3, \dots, 20$ and obtain the following graph for the mean squared error, which appears to be steadily decreasing as $N$ increases.
%
%\begin{center}
%\includegraphics[width=100mm]{Meansquarederror.png}
%\end{center}
%
%We then pick $N =12 $ since it is the first time the mean squared error is of the order $10^{-6}$. Using this, we get the best fit curve to be $f(x) = \frac{13.6505}{5.7872+x}$ with a mean squared error of $7.8 \times 10^{-6}$.  
%

\subsection{Total counts of surfaces in $\calH(2)$}

Recall that in $\calH(2)$, a square-tiled surface can either have one horizontal cylinder or two.  (Of course the same analysis applies to vertical cylinders,
but we will consider horizontal cylinders as a convention in this section.) We note here that the following asymptotics can be deduced from Zorich's work \cite{zorich} where he computes the asymptotics for the number of STSs upto $n$ squares. Here we present an alternate treatment by computing the asymptotics of the number of STSs with exactly $n$ squares.

Let us first estimate the number of $n$-square one-cylinder surfaces in $\calH(2)$.

Our estimate will involve the sigma  function $\sigma_x(n) = \sum\limits_{d | n} d^x$. The following is a useful observation. 

\begin{proposition} \label{sigma}
For $x > 1$, $\sigma_x(n)$ is $\Theta(n^x)$. 
\end{proposition}
\begin{proof} We want to show that $\sigma_x(n)$ is upper and lower bounded by a constant times $n^x$. For the lower bound, we have that $$\sigma_x(n) = \sum_{d|n} d^x \geq n^x.$$ For the upper bound, we have that $$\sigma_x(n) = \sum_{d|n} d^x \leq \sum_{k=1}^n \left(\frac{n}{k}\right)^x = n^x \sum_{k=1}^n \left(\frac{1}{k^x}\right)  \leq \zeta(x) n^x.$$
For $x > 1$, $\zeta(x)$ is finite and this gives us our upper bound. 
\end{proof} 

We return now to counting one-cylinder surfaces in $\calH(2)$. We refer again to Figure \ref{fig:H2surfaces}. Given that the squares in our square-tiled surface have unit side lengths, the length of the base is $k+l+m$, the height is $p$, and the shear $\alpha \in \{0,1,\ldots, k+l + m - 1\}$. The left side is identified to the right, and the top and bottom edges are broken up into three sub-edges with $h$ identified to $h$, and so on. Using the parameters, we get the following lemma.

\begin{lemma} 
\label{1cyltotal}
The number of $n$-square one-cylinder surfaces in $\calH(2)$ is $\Theta(n^3)$. 
\end{lemma} 

\begin{proof}
For each $r|n$ such that $r>2$, there are a positive number of square-tiled surfaces with horizontal saddle connections of length $k,l,m$ as in Figure \ref{fig:H2surfaces}, with $k+l+m = r$. We can estimate this number by noting that  for any $r$, there are $O(r^2)$ ways to partition $r$ into three numbers $k+l+m$, where order doesn't matter. This is because there are $\binom{r-1}{2}$ ways to write $r= k+l+m$ where order does matter, and then each triple $(k,l,m)$ has at most six possible reorderings. For each $k+l+m=r$, we can estimate the number of distinct shears $\alpha$. 

We know that there are at most $r$ distinct shears $\alpha$. Furthermore, for any shear $\alpha$ that gives an equivalent surface to the $\alpha = 0$ shear, we must have that a saddle connection vector of $(0,p)$ exists. But there are three singularities on the top edge of the cylinder, and three on the bottom, so there are at most $9$ different ways for which a top singularity to line up with a bottom singularity. Three of those pairings occur for $\alpha = 0$. So there are at most $6$ other shears that are equivalent to the $\alpha = 0$ shear. This gives us a coarse bound of at least $r/6$ distinct shears. (In fact, with a little more work, we can show that there are $r/3$ distinct shears when $k=l=m$ and $r$ distinct shears otherwise.)

This gives us $O(r)$ different $\alpha$ values. Once $r$ is given, $p = n/r$ is determined. In total, once we sum over all $r |n$, we have $\Theta(\sigma_3(n))$ distinct $n$-square one-cylinder surfaces, which is also $\Theta(n^3)$ by Proposition \ref{sigma}. 
\end{proof}

Now, we turn our attention to two-cylinder surfaces. We can first explicitly compute the total number of surfaces of this type. 

\begin{lemma} 
\label{2cyltotal}
The number of two-cylinder surfaces in $\STS_n \cap \calH(2)$ is $\frac{5}{24} \sigma_3(n) + \frac{1}{2} \sigma_2(n) - \frac{3}{4} n \sigma_1(n) + \frac{1}{24} \sigma_1(n)$.
\end{lemma}
\begin{proof}
The surfaces we are counting are uniquely described by their 6-tuples $(p,q,k,l,\alpha, \beta)$ with $p, q \in \NN$, $k > l \in \NN$ and $\alpha \in \{0, \dots, k-1\}$ and $\beta\in \{0, \dots, l-1\}$ such that 
$pk + ql = n$ as in the figure. 
We compute:
$$\sum_{\substack{p, q, k,l \in \NN \\  k > l \\ pk + ql = n}}\sum_{\alpha = 0}^{k-1}\sum_{\beta = 0}^{l-1} 1 = \sum_{\substack{p, q, k,l \in \NN \\  k > l \\ pk + ql = n}} kl $$
We proceed using symmetry between $k$ and $l$ and the substitution $P=pk$, $Q=ql$.
\begin{align*}
\sum_{\substack{p, q, k,l \in \NN \\  k >l \\ pk + ql = n}} kl & = \frac{1}{2}\left(\sum_{\substack{p, q, k,l \in \NN \\ pk + ql = n}} kl - \sum_{\substack{p, q, k \in \NN \\ (p + q)k = n}} k^2  \right) = \frac{1}{2}\left(\sum_{\substack{P, Q \in \NN\\ P + Q = n }}\sum_{k|P}k \sum_{l | Q}l - \sum_{k| n}k^2\left(\frac{n}{k}-1\right)\right)\\
&= \frac{1}{2}\left(\sum_{P = 1}^{n-1} \left(\sigma_1(P) \sigma_1(n-P)\right) - n\sdot \sigma_1(n) + \sigma_2(n)\right).
\end{align*}
From a well-known identity due to Ramanujan (\cite{ramanujan}), we have that $$\sum_{P = 1}^{n-1} \sigma_1(P) \sigma_1(n-P) = \frac{5}{12}\sigma_3(n) + \frac{1}{12}\sigma_1(n) - \frac{1}{2}n \sigma_1(n).$$
Hence, our sum becomes
$$\sum_{\substack{p, q, k,l \in \NN \\  k > l \\ pk + ql = n}} kl = \frac{5}{24}\sigma_3(n) + \frac{1}{2}\sigma_2(n) + \frac{1}{24}\sigma_1(n) - \frac{3}{4}n\sigma_1(n).$$\end{proof}

These two results together give us that the number of $n$-square surfaces in $\calH(2)$ is $\Theta(n^3)$.

\subsection{Unit saddles in $\calH(2)$}

In this section, we will prove that the number of $n$-square surfaces in $\calH(2)$ with a unit saddle is $\Theta(n^2)$. We will show this by first proving that the number of $n$-square surfaces in $\calH(2)$, either with one cylinder or with two cylinders, that have unit-length horizontal saddle connection is $\Theta(n^2)$. By symmetry, the number of $n$-square surfaces in $\calH(2)$ that have a unit-length 
vertical saddle connection is also $\Theta(n^2)$. These two asymptotics together give us our desired asymptotic for the number of $n$-square surfaces in $\calH(2)$ with a unit-length horizontal saddle connection. 

\begin{lemma}
\label{1cylunit} The number of $n$-square one-cylinder surfaces in $\calH(2)$ with a unit-length horizontal saddle connection is $\Theta(n^2)$.
\end{lemma} 

\begin{proof}
In this proof, $k,l,m$ are the horizontal saddle connection lengths, $p$ is the height, and $\alpha$ is the shear, as in Figure \ref{fig:H2surfaces}. For any $r|n$ with $r \geq 3$, it is possible to have a unit-length horizontal saddle connection. The shear $\alpha$ can be anything but we need at least one of $k, l, m$ to be $1$. The number of ways for $k+l+m=r$ where each of $k,l,m \geq 2$ is equal to $\binom{r-4}{2}$. Therefore, there are $\binom{r-1}{2} - \binom{r-4}{2} = \Theta(r)$ ways of choosing $k,l,m$ with at least one being $1$. Accounting for our choice for $\alpha$ as well, we get $\Theta(r^2)$ surfaces with a unit-length horizontal saddle connection. 

Note that all asymptotics computed here are in terms of universal constants, independent of $r$ and $n$. Putting this all together, we see that the number of surfaces with a unit-length 
saddle connection is $\Theta(\sigma_2(n))$, which is also $\Theta(n^2)$ by Lemma \ref{sigma}.
\end{proof}

\begin{lemma} 
\label{2cylunit}
The number of $n$-square two-cylinder surfaces in $\calH(2)$ with a unit-length horizontal saddle connection is $\Theta(n^2)$. 
\end{lemma}
\begin{proof} 

Recall that we can identify a two-cylinder surface by its cylinder widths $l$ and $k$, its shears $\beta$ and $\alpha$, and its cylinder heights $q$ and $p$, as in Figure \ref{fig:H2surfaces}. To have a unit-length horizontal saddle connection, our square-tiled surfaces must have either $l=1$ or $k-l = 1$. Then, given any shears $\alpha$ and $\beta$, the sheared square-tiled surface still have this unit-length horizontal saddle connection. The number of $n$-square two-cylinder surfaces with a unit-length horizontal saddle connection is then 
\begin{equation}
\label{eq:2cyl}
\sum_{\substack{pk + lq = n, k > l \\ l = 1 \text{ or }k-l =1}} kl = \sum_{pk + q = n, k > 1} k + \sum_{p(1+l) + lq = n} l(l+1) - \sum_{2p + q = n} 2.
\end{equation}

Let us examine each of the terms on the right hand side of equation \eqref{eq:2cyl}. The first term is $$\sum_{pk+q=n, k>1} k = \sum_{k=2}^n k \left\lfloor \frac{n}{k} \right\rfloor$$

Each term $k \left\lfloor \frac{n}{k} \right\rfloor$ is bounded above by $n$ and below by $n/2$, and so $$\frac{n(n-1)}{2} = \sum_{k=2}^n \frac{n}{2} \leq \sum_{k=2}^n k \left\lfloor \frac{n}{k} \right\rfloor \leq \sum_{k=2}^n n = n(n-1).$$

Therefore, the first term is $\Theta(n^2)$. 

We will now prove that the second term in equation \eqref{eq:2cyl} is $O(n^2)$. We first split the sum up as 
$$\sum_{p(1+l) + lq = n} l(l+1) = \sum_{l=1}^{\lfloor \sqrt{n} \rfloor} |\{(p,q) : p(l+1) + ql = n\}|l (l+1) + \sum_{l = \lfloor \sqrt{n} \rfloor + 1}^n |\{(p,q) : p(l+1) + ql = n\}| l(l+1).$$ 

For the first term, $|\{(p,q) : p(l+1) + ql = n\}|$ is bounded above by $\lfloor n/(l+1)\rfloor$, the number of possible positive $p$'s for which $p(l+1) \leq n$. Therefore, 
$$\sum_{l=1}^{\lfloor \sqrt{n} \rfloor} |\{(p,q) : p(l+1) + ql = n\}|l (l+1) \leq \sum_{l=1}^{\lfloor \sqrt{n} \rfloor} nl  =\frac{n\lfloor \sqrt{n} \rfloor(\lfloor \sqrt{n} \rfloor+1)}{2},  $$ which is $O(n^2)$. 

For the second term, we are summing $|\{(p,q) : p(l+1) + ql = n\}| l(l+1)$ from $l = \lfloor \sqrt{n} \rfloor +1$ to $l = n$. We can rewrite the first factor as $|\{(p,q) : (p+q)l + p = n\}|$. When $l \geq  \lfloor \sqrt{n} \rfloor +1 \geq \sqrt{n}$, we have that $(p+q) \leq \sqrt{n}$ and so $p < l$. Let $k = p+q$. Then, $|\{(p,q) : (p+q)l + p = n\}| \leq 1$, since for any given $l$, there is then at most one value of $k$ for which $kl + p = n$ has a positive integral solution. Then, $p+q = k$ and $p = n-kl$ determine $p$ and $q$. 

Furthermore, if we are trying to solve $kl + p = n$ with positive integral values, where $k$ is fixed and $p < k = p+q$, there is at most one value of $l$ for which $kl + p = n$ has a solution, and this $l$ must satisfy $l \leq n/k$. Therefore, we can upper bound our sum by indexing our sum over possible values of $k$:

$$\sum_{l = \lfloor \sqrt{n} \rfloor + 1}^n |\{(p,q) : p(l+1) + ql = n\}| l(l+1) \leq \sum_{k=1}^{\lfloor \sqrt{n} \rfloor} \frac{n}{k} \left( \frac{n}{k} + 1 \right)  = n^2 \sum_{k=1}^{\lfloor \sqrt{n} \rfloor} \frac{1}{k^2} + \sum_{k=1}^{\lfloor \sqrt{n} \rfloor} \frac{n}{k} \leq \zeta(2) n^2 + n \sqrt{n},$$
which is $O(n^2)$. 

Finally, the third term of equation \eqref{eq:2cyl} is $2$ times $|\{2p+q = n\}|$, which is $2 \lfloor (n-1)/2 \rfloor$ and is $\Theta(n)$. 

We have thus proved that the three terms in \eqref{eq:2cyl} are $\Theta(n^2), O(n^2),$ and $\Theta(n)$ respectively, which is enough to show that count of $n$-square two-cylinder surfaces in $\calH(2)$ with a unit-length horizontal saddle connection is $\Theta(n^2)$, as desired. 
\end{proof} 

Now we can deduce the asymptotics for those $n$-square surfaces in $\calH(2)$ with a unit saddle. We can combine these with our counting results for STSs in $\mathcal{H}(2)$ to prove our theorem.

\begin{proof}[Proof of Theorem \ref{total}] The number of one-cylinder surfaces in $\STS_n \cap \calH(2)$ is $\Theta(n^3)$ by Lemma \ref{1cyltotal}. From Lemma \ref{2cyltotal}, we have that the number of two-cylinder surfaces in $\STS_n \cap \calH(2)$ is $$\frac{5}{24}\sigma_3(n) + \frac{1}{2}\sigma_2(n) + \frac{1}{24}\sigma_1(n) - \frac{3}{4}n\sigma_1(n).$$
We know that $\sigma_1(n)$ is bounded below by $n$ and above by $1^2 + 2^2 + \ldots + n^2 = \Theta(n^3)$. From this and Lemma \ref{sigma}, we have that the number of two-cylinder surfaces in $\STS_n \cap \calH(2)$ is also $\Theta(n^3)$. Putting the asymptotics for one-cylinder and two-cylinder surfaces together gives us that the total number of square-tiled surfaces in $\mathcal{H}(2)$ is also $\Theta(n^3)$. 

For the counts of surfaces in $\mathcal{H}(2)$ with unit saddles, we know that $|\STS_n \cap \calH(2) \cap \Unit |$ is at least the number of surfaces in $\STS_n \cap \calH(2)$ with a unit-length horizontal saddle connection. From Lemmas \ref{1cylunit} and \ref{2cylunit}, we know that this lower bound is $\Theta(n^2)$. To prove an upper bound, we first observe that by symmetry, the number of surfaces in $\STS_n \cap \calH(2)$ with a unit-length vertical saddle connection is equal to the number with a unit-length horizontal saddle connection. Therefore, $|\STS_n \cap \calH(2) \cap \Unit |$ is upper-bounded by two times the number of surfaces in $\STS_n \cap \calH(2)$ with a unit-length horizontal saddle connection, which is $\Theta(n^2)$. The upper and lower bounds together show that  $|\STS_n \cap \calH(2) \cap \Unit |$ is $\Theta(n^2)$. 
\end{proof}

 \section{Questions and future work}
 Many interesting questions remain.  
\begin{enumerate}
%\item Give asymptotics for $|\STS_n|$, $|\RSTS_n|$, $|\PSTS_n|$ as $n\to\infty$.
\item Are there polynomial bounds (or any reasonable bounds) with respect to genus for the counts of each type of fake torus in each stratum? 
Do symmetry, holonomy, and visibility tori have well-defined asymptotic density in $\RSTS_n$?
\item Symmetry tori seem to be quite rare.  Give a classification of symmetry tori (perhaps some families and some sporadic examples).
\item Can any reduced STS have {\em no} relatively prime vectors in its holonomy?  (Recall that if this is impossible, then symmetry tori are always visibility tori, completing the 
last missing implication in Figure~\ref{fig:implications}.)
\item For each fixed genus $g$, there are finitely many visibility tori.
How are these distributed among the strata?  For example, are most of them in the principal stratum $\calH(1,1,\dots,1)$?
\item Recall that there is a bound $M(\alpha)$ giving the most squares $n$ for which there exists a visibility torus in $\calH(\alpha)$.  Give an explicit formula in terms of 
$\alpha$.
\item Do similar asymptotics as those for unit saddles in $\calH(2)$ hold for other strata? How about if we consider the saddle connections of other lengths greater than 1?
\end{enumerate}

%
%\section{Examples of visibility tori and non-visibility STSs}
%
%One of our eventual goals is to make explicit the constant $M(\alpha)$, which denotes the largest number of squares a visibility torus is allowed to have in a stratum $\calH(\alpha)$. As a step towards this goal, we search for large examples of visibility tori in each stratum, and in \ref{fig:largesmallexamples} we list them for some early strata, along with the smallest non-visibility STS in each of these strata.

\newpage

\appendix

\section{Parametrization of genus two STSs}\label{appendix:parametrization}

Knowing that STSs in $\calH(2)$ have either one or two horizontal cylinders, we can parametrize the two types of square-tiled surfaces in $\calH(2)$ as in Figure \ref{fig:H2surfaces}. For STSs in $\calH(1,1)$, we give the parametrization in Figure \ref{fig:H11surfaces}.

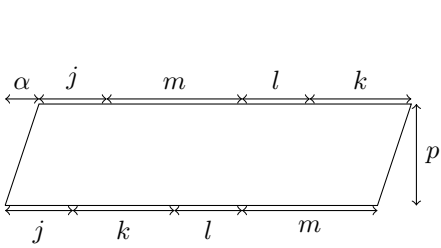
\begin{figure}[h!!!!!!]
%\vspace{-15pt}
\begin{tabular}{ccc}
\begin{tikzpicture}[scale=.6]
\draw (0,0) -- (6,0) -- (7,3) -- (1,3) -- (0,0); 
\draw [<->] (0,-0.15) -- (2, -0.15);
\node at (1,-0.35) {$k$};
\draw [<->] (2,-0.15) -- (3, -0.15);
\node at (2.5,-0.35) {$l$};
\draw [<->] (3,-0.15) -- (6, -0.15);
\node at (4.5,-0.35) {$m$};

\draw [<->] (1,3.15) -- (3, 3.15);
\node at (2,3.35) {$k$};
\draw [<->] (3,3.15) -- (6, 3.15);
\node at (4.5,3.35) {$m$};
\draw [<->] (6,3.15) -- (7, 3.15);
\node at (6.5,3.35) {$l$};

\draw [<->] (7.15, 0) -- (7.15, 3);
\node at (7.35, 1.5) {$p$};

\draw [<->] (0, 3.15) -- (1, 3.15);
\node at (0.5, 3.35) {$\alpha$};
\end{tikzpicture} & \hspace{1cm} &

\begin{tikzpicture}[scale=.5]
\draw (0,0) -- (6,0) -- (7,3) -- (1,3) -- (0,0);
\draw (1,3) -- (3,5) -- (6,5) -- (4,3);
\draw [<->] (0,-0.15) -- (6, -0.15);
\node at (3,-0.35) {$k$};
\draw [<->] (3,5.15) -- (6, 5.15);
\node at (4.5,5.45) {$l$};
\draw [<->] (7.15, 0) -- (7.15, 3);
\node at (7.35, 1.5) {$p$};
\draw [<->] (6.15, 3) -- (6.15, 5);
\node at (6.35, 4) {$q$};
\draw [<->] (0,3.15) -- (1, 3.15);
\node at (0.5,3.35) {$\alpha$};
\draw [<->] (1,5.15) -- (3, 5.15);
\node at (2,5.35) {$\beta$};
\end{tikzpicture} \\
 $k,l,m,p \in \NN$ are arbitrary & \hspace{1cm} & $k,l,p,q\in\NN$ are arbitrary, \\
 and $0\le \alpha < k+l+m$ are integer & \hspace{1cm} & and 
$0\le \alpha \le k-1$ and $0\le \beta\le l-1$ are integers.
\end{tabular}

\caption{Parametrization of square-tiled surfaces in $\calH(2)$}
\label{fig:H2surfaces}
\end{figure}

\begin{figure}[ht]
\begin{tabular}{ccc}
\begin{tikzpicture}[scale=.45]
\draw (0,0) -- (11,0) -- (12,3) -- (1,3) -- (0,0); 
\draw [<->] (0,-0.15) -- (2, -0.15) node[midway, below=0.15] {$j$};
\draw [<->] (2,-0.15) -- (5, -0.15) node[midway, below=0.15] {$k$};
\draw [<->] (5,-0.15) -- (7, -0.15) node[midway, below=0.15] {$l$};
\draw [<->] (7,-0.15) -- (11, -0.15) node[midway, below=0.15] {$m$};

\draw [<->] (1,3.15) -- (3, 3.15) node[midway, above=0.15] {$j$};
\draw [<->] (3,3.15) -- (7, 3.15) node[midway, above=0.15] {$m$};

\draw [<->] (7,3.15) -- (9, 3.15) node[midway, above=0.15] {$l$};

\draw [<->] (9,3.15) -- (12, 3.15) node[midway, above=0.15] {$k$};

\draw [<->] (12.15, 0) -- (12.15, 3) node[midway, right=0.15] {$p$};

\draw [<->] (0, 3.15) -- (1, 3.150) node[midway, above=0.15] {$\alpha$};

\end{tikzpicture} & \hspace{1cm} &

\begin{tikzpicture}[scale=.45]
\draw (0,0) -- (11,0) --(12, 2) -- (4,2) -- (6,5) -- (3,5) -- (1,2) -- (0,0);
\draw [<->] (0, -0.15) -- (2+1, -0.15) node[midway, below=0.15] {$m$};
\draw [<->] (2+1, -0.15) -- (5+1, -0.15) node[midway, below=0.15] {$l$};
\draw [<->] (5+1, -0.15) -- (9+2, -0.15) node[midway, below=0.15] {$k$};

\draw [<->] (3.1+1, 0.15+2) -- (7+2, 0.15+2) node[midway, above=0.15] {$k$};
\draw [<->] (7+2, 0.15+2) -- (10+2, 0.15+2) node[midway, above=0.15] {$l$};

\draw [<->] (3, 0.15+2+3) -- (5+1, 0.15+2+3) node[midway, above=0.15] {$m$};

\draw [<->] (0, 0.15+2) -- (1, 0.15+2) node[midway, above=0.15] {$\alpha$};
\draw [<->] (1, 0.15+2+3) -- (3, 0.15+2+3) node[midway, above=0.15] {$\beta$};

\draw [<->] (12.15, 0) -- (12.15, 2) node[midway, right=0.15] {$p$};
\draw [<->] (12.15, 2) -- (12.15, 5) node[midway, right=0.15] {$q$};

\end{tikzpicture} \\

Type A : $j, k, l, m, p \in \NN$ arbitrary  & \hspace{1cm} & Type B: $k, l, m, p, q \in \NN$ arbitrary \\
and $0 \leq \alpha < j+k+l+m$ &  \hspace{1cm} & and $0 \leq \alpha < k+l+m$, $0 \leq \beta < m$.\\[10pt]

\begin{tikzpicture}[scale=.45]
\draw (2,2) -- (5,2) --(4, 0) -- (9,0) -- (10,2) -- (8,2) -- (10,5) -- (4,5) -- (2,2);

\draw [<->] (2, 2-0.15) -- (4.9, 2-0.15) node[midway, below=0.15] {$m$};
\draw [<->] (4, -0.15) -- (7, -0.15) node[midway, below=0.15] {$l$};
\draw [<->] (7, -0.15) -- (9, -0.15) node[midway, below=0.15] {$k$};

\draw [<->] (2+2, 2+0.15+3) -- (5+2, 2+0.15+3) node[midway, above=0.15] {$m$};
\draw [<->] (4+3, 0.15+5) -- (7+3, 0.15+5) node[midway, above=0.15] {$l$};
\draw [<->] (7+1.1, 0.15+2) -- (9+1, 0.15+2) node[midway, above=0.15] {$k$};

\draw [<->] (2, 2+0.15+3) -- (2+2, 2+0.15+3) node[midway, above=0.15] {$\beta$};
\draw [<->] (9, -0.15) -- (10, -0.15) node[midway, below=0.15] {$\alpha$};

\draw [<->] (10.15, 0) -- (10.15, 2) node[midway, right=0.15] {$p$};

\draw [<->] (10.15, 2) -- (10.15, 5) node[midway, right=0.15] {$q$};

\end{tikzpicture} & \hspace{1cm} &

\begin{tikzpicture}[scale=.45]
\draw (1,2) -- (4,2) -- (2,0) --(5, 0) -- (7,2) -- (8,4) -- (5,4) -- (7,7) -- (4,7) -- (2,4) -- (1,2);

\draw [<->] (2, -0.15) -- (5, -0.15) node[midway, below=0.15] {$k$};
\draw [<->] (1, 2-0.15) -- (3.9, 2-0.15) node[midway, below=0.15] {$l$};

\draw [<->] (2+2+1+0.1, 0.15+2+2) -- (5+2+1, 0.15+2+2) node[midway, above=0.15] {$k$};
\draw [<->] (1+1+3-1, 0.15+2+3+2) -- (4+1+3-1, 0.15+5+2) node[midway, above=0.15] {$l$};

\draw [<->](1,0.15+2+2)  --(2+2+1-3, 0.15+2+2) node[midway, above=0.15] {$\beta$};
\draw [<->] (2, 0.15+2+3+2) -- (5-1, 0.15+5+2) node[midway, above=0.15] {$\gamma$};
\draw [<->] (5, -0.15) -- (7, -0.15) node[midway, below=0.15] {$\alpha$};

\draw [<->] (8.15, 0) -- (8.15, 2) node[midway, right=0.15] {$p$};
\draw [<->] (8.15, 2) -- (8.15, 4) node[midway, right=0.15] {$q$};
\draw [<->] (8.15, 4) -- (8.15, 7) node[midway, right=0.15] {$r$};

\end{tikzpicture}\\

Type C: $k, l, m, p, q \in \NN$ arbitrary &  \hspace{1cm} & Type D: $k, l, p, q, r \in \NN$ arbitrary \\
and $0 \leq \alpha < k+l$ and $0 \leq \beta < l+m$ &  \hspace{1cm} & and $0 \leq \alpha < k$, $0 \leq \beta < k+l$, $0 \leq \gamma < l$.

\end{tabular}
\caption{Parametrization of square-tiled surfaces in $\calH(1,1)$.}
\label{fig:H11surfaces}
\end{figure}

%\section{Using surface dynamics to find fake tori}

\end{document}